\documentclass[12pt]{amsart}
\usepackage{amssymb,latexsym,amsmath,enumitem,mathrsfs,geometry}
\usepackage{thmtools}
\usepackage{fullpage}
\usepackage{fancyhdr}
\usepackage{xcolor}
\usepackage{bbm}
\usepackage{stmaryrd}
\usepackage{mathrsfs}
\usepackage{lineno}
\usepackage{diagbox}
\usepackage{array}
\usepackage{tikz}
\usetikzlibrary{arrows}
\usetikzlibrary{positioning, arrows.meta, shapes.geometric}
\usepackage{float}
\usepackage{comment}
\usepackage{mathtools}
\usepackage{amsthm}
\usepackage{capt-of}
\usepackage{hyperref} 
\hypersetup{
    colorlinks=false,
    pdfborder={0 0 1},
    linkbordercolor={1 0 0},
    citebordercolor={0 1 0},
    urlbordercolor={0 1 1},
    filebordercolor={1 0 1},
    pdfpagemode=FullScreen}

\newtheorem{theorem}{Theorem}[section]
\newtheorem*{theorem*}{Theorem}

\newtheorem*{remark*}{Remark}

\newtheorem{corollary}[theorem]{Corollary}
\newtheorem{proposition}[theorem]{Proposition}

\theoremstyle{remark}
\newtheorem{definition}[theorem]{Definition}
\newtheorem{remark}[theorem]{Remark}
\theoremstyle{remark}
\newtheorem{example}[theorem]{Example}

\begin{document}

\numberwithin{equation}{section}

\pagestyle{plain}

\begin{center}
{\normalsize\bfseries SEGMENTS AND CONVEXITY IN METRIC SPACES\par}

\vspace{1.2em}

{\small TIAN VLASIC\par}

\vspace{0.15em}

{\footnotesize

Department of Mathematics, Harvard University\par

Cambridge, MA 02138, USA\par

\href{mailto:tianvlasic@college.harvard.edu}{tianvlasic@college.harvard.edu}\par

}

\vspace{0.3em}

{\footnotesize \today\par}

\end{center}
\vspace{0.8em}

\begin{center}
\begin{minipage}{0.85\textwidth}
\small
\noindent\textbf{Abstract.}
This paper aims both to provide a unified introduction to $d$-convexity and to contribute new structural results on metric segments and convexity in metric spaces. After developing the basic theory of metric segments, we study their geometric and topological properties, establishing several structural results, and show that metric segments realize arbitrary closed subsets after the metric is replaced by a topologically equivalent and bounded one. We then examine metrically, Menger, and strictly convex spaces, obtaining new characterizations of strict convexity in terms of an order structure on metric segments and their isometric embeddability into $\mathbb{R}$. Additionally, we offer an exposition of the relationship between $d$-convexity and other notions of convexity in metric spaces from the standpoint of axiomatic convexity. Finally, we characterize metric spaces in which metric segments are trivial via a "local snowflaking" condition.
\end{minipage}
\end{center}

\vspace{1em}

\begin{center}
\begin{minipage}{0.85\textwidth}

\medskip


\end{minipage}

\end{center}

\vspace{1em}

\section{Introduction} \label{sec1}
Convex sets are mainly studied in the context of  real vector spaces and are defined as subsets closed under convex combinations of their elements. Equivalently, a subset of a real vector space is convex if, given any pair of its elements, it contains the \emph{line segment} joining them. \emph{Linear convexity} is foundational for many fields, including optimization, functional analysis, and convex geometry. However, there are many structures other than real vector spaces for which it is natural and useful to define analogous notions of convexity of sets: notable examples include partially ordered sets, where a subset $C \subseteq X$ of a partially ordered set $(X,\preceq)$ is called convex if $z \in C$ holds for all $x,y \in C$ and $z \in X$ with $x \preceq z \preceq y$, and, importantly to us, metric spaces.\\

\noindent There are a few inequivalent ways in which one can define a notion of convexity of sets in metric spaces (this will be discussed in more detail in Section~\ref{sec4}). That being said, this paper favors \emph{$d$-convexity} as the most natural notion of convexity of sets in metric spaces, which was first introduced by Menger \cite{menger1928untersuchungen} in 1928. Specifically, given a metric space $(X,d)$ and points $x,y$, Menger first defines the \emph{metric segment} joining $x,y$ as the set

$$[x,y]= \left\{ z \in X: d(x,z)+d(z,y)=d(x,y)\right\}.$$

\vspace{0.3cm}

\noindent (A useful intuition for the reader is that metric segments are degenerate cases of the triangle inequality). For the sake of terminology, we will refer to the points $x,y$ as the \emph{endpoints} of the metric segment $[x,y]$. Also, when there is ambiguity, we will specify the metric or space in which the metric segments are being considered via a subscript. Now, we say that a subset $C \subseteq X$ is \emph{$d$-convex} (hereafter, simply, \emph{convex}) if $[x,y] \subseteq C$ holds for all $x,y \in C$. With this definition, Menger mimics line segments in real vector spaces and linear convexity. In fact, $d$-convexity and linear convexity coincide in a large class of real normed spaces called \emph{strictly convex} spaces, which are normed spaces for which the unit closed ball is a strictly convex set (the class of strictly convex spaces is known to include, among other spaces, real inner product spaces). More specifically, we have the following proposition:\\

\begin{proposition}\label{prop1.1}
Let $V$ be a normed  space. Then, $V$ is strictly convex if and only if $[u,v]=\left\{ (1-t)u+tv: t \in [0,1]\right\}$ holds for all $u,v \in V$.
\end{proposition}

\vspace{0.3 cm}

\begin{proof}
See Propositions 3.6 and 4.1 in \cite{bankston2018semicontinuity}.
\end{proof}

\noindent \\ As an immediate corollary, we get: \\

\begin{corollary}
A subset of a strictly convex space is $d$-convex if and only if it is linearly convex.
\end{corollary}

\vspace{0.3 cm}

\noindent Conversely, a beautiful illustration of the failure of Proposition~\ref{prop1.1} to hold for normed  spaces that are not strictly convex comes from considering $\ell^p$ norms on $\mathbb{R}^2$. \\

\begin{figure}[H]
\centering
\includegraphics[width=9 cm]{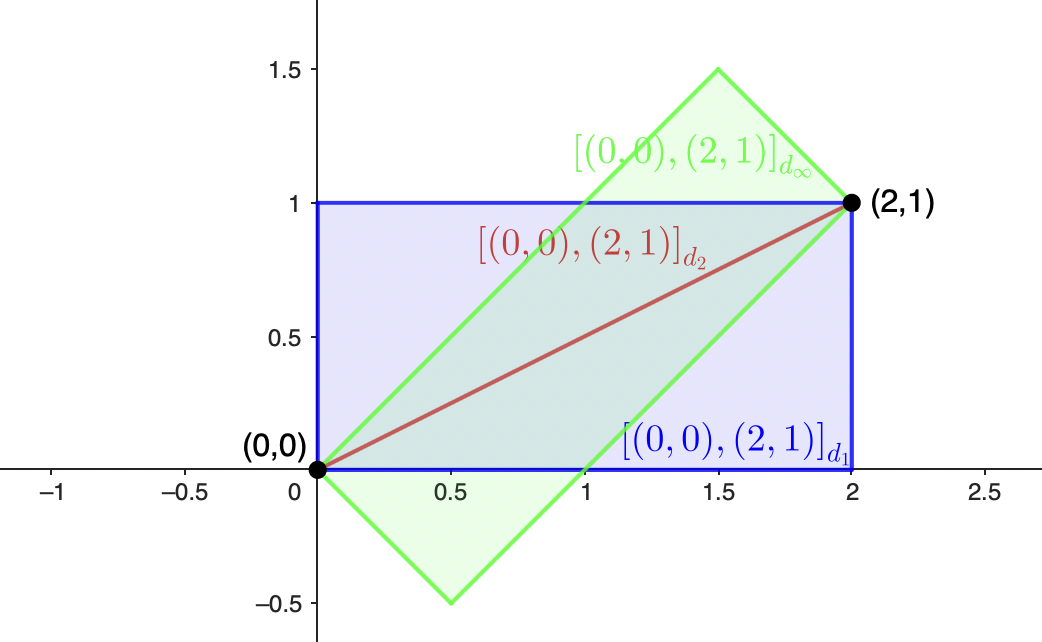} 
\caption{Metric segments in $\mathbb{R}^2$ endowed with the $\ell^1$, $\ell^2$, and $\ell^{\infty}$ norms.}
\label{fig1}
\end{figure}

\noindent \\In fact, $\ell^p$ norms on $\mathbb{R}^n$ are known to not be strictly convex precisely when $p \in \left\{ 1, \infty \right\}$. Since Menger, $d$-convexity has been studied by many authors (see e.g., \cite{hertel1994convexity}, \cite{soltan1984introduction}).\\

On the other hand, recently, we have witnessed a lot of work develop in the direction of attempting to formulate a unified framework under which the notions of convexity of sets in the various structures we have mentioned (real vector spaces, partially ordered sets, and metric spaces) can be studied. A prominent theory that has emerged from these efforts is that of \emph{convexity spaces}. Specifically, a convexity space is an ordered pair $(X, \mathcal{C})$ consisting of a set $X$ and a collection $\mathcal{C}\subseteq \mathcal{P}(X)$ of subsets of $X$ called a \textit{convexity} on $X$ satisfying the following properties:\\

\begin{enumerate}
\itemsep0.3em
\item[(C1)] $\emptyset, X \in \mathcal{C}$. \label{c1}
\item[(C2)] The intersection of an arbitrary collection of elements of $\mathcal{C}$ is an element of $\mathcal{C}$. \label{c2}
\item[(C3)] The union of an arbitrary chain under inclusion of elements of $\mathcal{C}$ is an element of $\mathcal{C}$. \label{c3} 
\end{enumerate} 

\noindent \\ An element of $\mathcal{C}$ is called a \textit{convex set}, and a set that is not convex is simply said to be  \textit{non-convex} (the term concave is discouraged for obvious reasons). The reader should note that the axiom \hyperref[c1]{(C1)} is redundant since it follows from the axioms \hyperref[c2]{(C2)} and \hyperref[c3]{(C3)} in the case when the collection of convex sets being considered is empty. Nevertheless, we include it in the definition of convexity spaces for the sake of clarity. This notion of convexity of sets is often referred to as \textit{axiomatic convexity}.\\

\noindent Convexity spaces indeed adequately generalize the properties of convex sets in real vector spaces, partially ordered sets, and metric spaces, as convex sets in all three structures are easily seen to satisfy axioms \hyperref[c1]{(C1)}-\hyperref[c3]{(C3)} and hence form a convexity. In particular, given a metric space $(X,d)$, we will denote the set of all convex subsets of $(X,d)$ by $\mathcal{C}_d$ and refer to it as the convexity \emph{induced} by $d$. Moreover, one can construct the category $\textbf{Conv}$ of convexity spaces if morphisms are chosen to be maps $f: (X, \mathcal{C}) \rightarrow (Y, D)$ satisfying $\forall C \in \mathcal{D} : f^{-1}(C) \in \mathcal{C}$. Then, we have the expected faithful functors $\textbf{Vect}_{\mathbb{R}} \rightarrow \textbf{Conv}$ and $\textbf{Pos} \rightarrow \textbf{Conv}$. On the other hand, if the category $\textbf{Met}$ of metric spaces is defined to have non-expansive or continuous maps as morphisms, then, unsurprisingly, we do not have an analogous faithful functor $\textbf{Met} \rightarrow \textbf{Conv}$. However, such a functor does exist when more natural morphisms are chosen for the category $\textbf{Met}.$\footnote{A more natural choice of morphisms for $\textbf{Met}$ making the category interact well with convexity is given by \emph{betweenness-preserving maps} i.e., maps $f:(X,d) \rightarrow (Y, \rho)$ satisfying $\forall x,y \in X: f([x,y]_d)\subseteq [f(x), f(y)]_{\rho}$ (terminology borrowed from \cite{hou2008betweenness}).}\\

\noindent Note that the axioms \hyperref[c1]{(C1)} and \hyperref[c2]{(C2)} ensure that every subset of a convexity space admits a unique smallest convex superset under inclusion, called its \emph{convex hull}. More formally, given a convexity space $(X, \mathcal{C})$ and a subset $S \subseteq X$, we can define its convex hull as the following set:\\

\vspace{-0.55 cm}

$$
\textup{conv}(S)=\bigcap \left \{ C : S \subseteq C \in \mathcal{C} \right \}.
$$

\vspace{0.2 cm}

\noindent Additionally, the comparison of convexities is established as follows. Given a pair of convexities $\mathcal{C}, \mathcal{D}$ on a set $X$, we say that $\mathcal{C}$ is \emph{coarser} than $\mathcal{D}$ i.e., that $\mathcal{D}$ is \emph{finer} than  $\mathcal{C}$ if $\mathcal{C} \subseteq  \mathcal{D}$. Axiomatic convexity was first introduced by Kay and Womble (see \cite{kay1971axiomatic}) in 1971 and has since been studied by many authors (see e.g., \cite{van1993theory}).

\vspace{0.5 cm}

This paper is organized as follows. Section~\ref{sec2} develops the basic theory of metric segments. We outline their fundamental geometric and topological properties, establish structural constraints on endpoints of metric segments and study the realizability of subsets as metric segments. Section~\ref{sec3} studies metrically, Menger, and strictly convex spaces. We examine the relationships between these spaces and give several equivalent characterizations of strict convexity. Finally, Section~\ref{sec4} compares $d$-convexity to other notions of convexity in metric spaces, including geodesic and $W$-convexity, and analyzes how and when the different notions of convexity agree and differ. The main contributions of this paper are:\\

\begin{itemize}
\itemsep0.3em
\item We show that metric segments realize arbitrary closed subsets after the metric is replaced by a topologically equivalent and bounded one (Theorem~\ref{thm2.10});
\item We characterize strict convexity in terms of an order structure on metric segments and their isometric embeddability into $\mathbb{R}$ (Theorem~\ref{thm3.3});
\item We characterize metric spaces with trivial metric segments via "local snowflaking" (Theorem~\ref{thm4.15}).
\end{itemize}

\vspace{0.3 cm}

\section{Metric Segments} \label{sec2}
Metric segments are the foundational and most basic concept in $d$-convexity. Thus, to effectively study $d$-convexity, it is of interest to first acquire a good understanding of the basic properties that metric segments satisfy. We begin this section by providing the reader with an additional visualization of metric segments to Figure~\ref{fig1}:

\vspace{0.2 cm}

\begin{example}
Consider the \emph{river metric} on $\mathbb{R}^2$:\\

$$d^*(p,q)=\begin{cases}
|y_1-y_2|, & x_1=x_2,\\[4pt]
|y_1|+|x_1-x_2|+|y_2|, & x_1\neq x_2.
\end{cases}$$\\

\noindent for points $p=(x_1,y_1), q=(x_2,y_2)$. One can easily verify that metric segments $[p,q]$ under $d^*$ are either vertical line segments (when $x_1=x_2$) or consist of a horizontal component on the $x$-axis and up to two vertical components (when $x_1 \neq x_2$), as illustrated by the following diagram:\\

\begin{center}
\includegraphics[width=10.5cm]{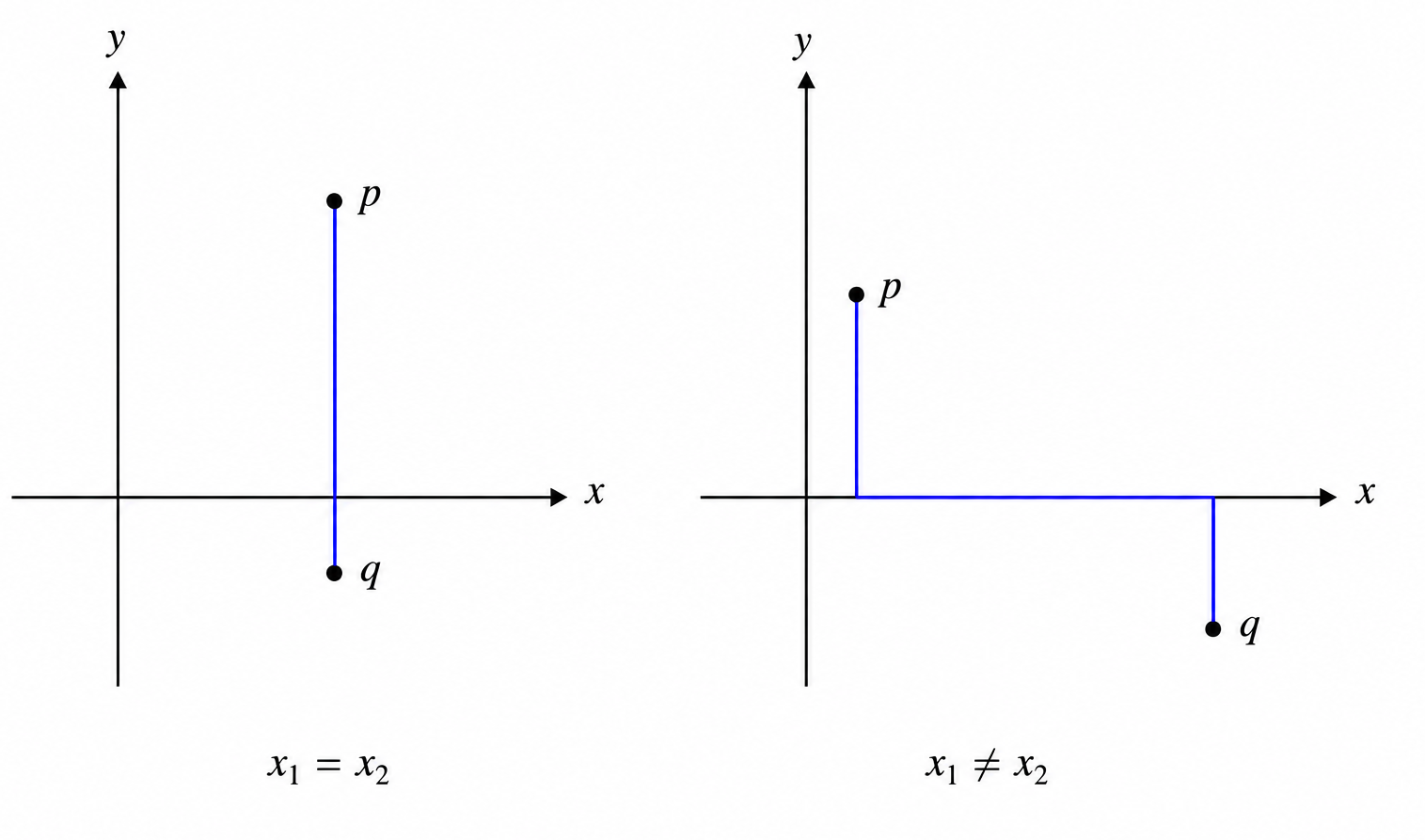}
\end{center}

\newpage

\captionof{figure}{Metric segments in $\mathbb{R}^2$ endowed with the river metric.}
\label{fig2}
\end{example}

\vspace{0.3 cm}

\noindent In \cite{okicic2023some}, Okičić and Rekić-Vuković study metric segments under the river metric in more depth. Furthermore, the following two propositions are well-known in literature and outline the most basic properties that general metric segments satisfy:\\

\begin{proposition}\label{prop2.2}
Let $(X,d)$ be a metric space and $x,y \in X$. Then, the following statements hold:\\

\begin{enumerate}
\itemsep0.3em
\item[\textup{($i$)}] $\left [ x,y \right ]$ is a closed set.
\item[\textup{($ii$)}] $\left [ x,y \right ]$ is a bounded set. Moreover, it holds that $\textup{diam}(\left [ x,y \right ])=d(x,y)$.
\end{enumerate}
\end{proposition}

\vspace{0.3 cm}

\begin{proof}
$(i)$ We show that $[x,y]$ includes all of its limit points. Namely, let $p_n \rightarrow p$ be a convergent sequence of elements of $[x,y]$. By definition, $d(x,p_n)+d(p_n,y)=d(x,y)$ holds for all $n \in \mathbb{N}$. Taking $n \to \infty$ and using continuity of the metric, we get that\\

$$d(x,y)=\displaystyle \lim_{n\to \infty}\left ( d(x,p_n)+d(p_n,y) \right )=d(x,p)+d(p,y).$$\\

\noindent Thus, $p \in [x,y]$, and $[x,y]$ is closed.\\

\noindent $(ii)$ Let $p,q \in \left [ x,y \right ]$. By definition, we have that\\

$$d(x,p)+d(p,y)+d(x,q)+d(q,y)=2d(x,y).$$

\noindent \\ Furthermore, triangle inequality implies \\

$$d(x,p)+d(p,y)+d(x,q)+d(q,y) \geq 2d(p,q).$$\\

\noindent Hence, $2d(x,y) \geq 2d(p,q)$ i.e., $d(p,q) \leq d(x,y)$. This implies that $\textup{diam}([x,y]) \leq d(x,y)$. On the other hand, one can easily verify that $x,y \in [x,y]$. Hence, $\textup{diam}([x,y]) \geq d(x,y)$. This completes the proof.
\end{proof}

\vspace{0.3 cm}

\noindent Despite the fact that metric segments are closed and bounded, we cannot, in general, expect metric segments to be compact, as opposed to line segments in normed vector spaces. This can be seen from the following example:\\

\begin{example}\label{ex2.3}
Consider $\mathbb{Q}$ under the Euclidean metric. We have that $[0,1]=I \cap \mathbb{Q}$, where $I$ is the unit closed interval in $\mathbb{R}$. However, $I \cap \mathbb{Q}$ is not compact (one can, for instance, see this by sequential compactness).
\end{example}

\vspace{0.3 cm}

\begin{proposition}\label{prop2.4}
Let (X,d) be a metric space. Then, the following statements hold:\\

\begin{enumerate}
\itemsep0.3em
\item[$(i)$] For all $x,y \in X$, we have $x,y \in [x,y]$. \hfill (reflexivity)
\item[$(ii)$]  For all $x,y \in X$, we have $[x,y]=[y,x]$. \hfill (symmetry)
\item[$(iii)$]  For all $x,y, z \in X$, we have that $z \in [x,y]$ implies $[x,z],[z,y] \subseteq [x,y]$. \hfill (transitivity)
\item[$(iv)$] For all $x,y \in X$ and $z \in [x,y]$, we have $[x,z]\cap[z,y]=\left \{z \right \}$.
\end{enumerate}
\end{proposition}

\vspace{0.5 cm}

\begin{proof}
The proofs of $(i)$ and $(ii)$ are left as an easy exercise for the reader.\\

\noindent $(iii)$ Let $p \in \left [ x, z \right ]$. The triangle inequality gives

\vspace{0.2 cm}

$$d(x,p)+d(p,y) \leq d(x,p)+d(p,z)+d(z,y).$$\\

\noindent Since $p \in \left [ x, z \right ]$, we know $d(x,p)+d(p,z)=d(x,z)$. Analogously, since $z \in [x,y]$, we get\\

$$d(x,p)+d(p,z)+d(z,y)=d(x,z)+d(z,y)=d(x,y)$$\\

\noindent Hence, $d(x,p)+d(p,y) \leq d(x,y)$. As we certainly have $d(x,p)+d(p,y) \geq d(x,y)$ by the triangle inequality, it follows that $d(x,p)+d(p,y)=d(x,y)$, which is equivalent to the statement $p \in [x,y]$. We conclude that $[x,z] \subseteq [x,y]$. By symmetry, we also conclude that $[z,y] \subseteq [x,y]$. \\

\noindent $(iv)$ Let $p \in [x,z] \cap [z,y]$. Then, by definition, the following relations hold\\

\begin{align*}
d(x,p)+d(p,z)&=d(x,z)\\
d(z,p)+d(p,y)&=d(z,y). 
\end{align*}

\noindent \\ Furthermore, by $(iii)$, we know that $p \in [x,y]$. Combining our observations, we get that\\

\begin{align*}
d(p,z)&=\frac{1}{2}\left ( d(x,z)-d(x,p)+d(z,y)-d(p,y) \right )\\
&=\frac{1}{2}(d(x,y)-d(x,y))=0.
\end{align*}

\noindent \\ Hence, $p=z$, and $[x,z] \cap [z,y]=\left\{ z\right\}$.
\end{proof}

\vspace{0.5 cm}

\noindent Note that every metric space $X$ induces a natural ternary relation $B$ on $X$ given by $\forall x,y,z \in X: B(x,z,y) \Leftrightarrow z \in [x,y]$, often called the \emph{betweenness relation} on $X$. In the language of this relation, Menger was the first to prove the property $(iii)$ in Proposition~\ref{prop2.4} (see \cite[Erste Untersuchung, §2]{menger1928untersuchungen}). Nowadays, betweenness in metric spaces continues to be an active area of research (see e.g., \cite{chvatal2004sylvester}). On the other hand, the property $(iv)$ in Proposition~\ref{prop2.4} does not seem to have a universally accepted terminology, but it appears extensively in literature on graph theory (see e.g., \cite{mulder1980interval}).\\

At first glance, the property $(ii)$ in Proposition~\ref{prop2.4} suggests that the endpoints of metric segments are unique up to permutation. More specifically, one might expect $[x,y]=[x',y'] \Rightarrow \left\{ x,y\right\}=\left\{ x',y'\right\}$ to hold for all points $x,y,x',y'$ in a metric space $X$. However, this is easily seen to be false by considering the $\ell^1$ and $\ell^{\infty}$ norms on $\mathbb{R}^2$ (see Figure~\ref{fig1}). For the sake of completeness, we provide an additional counterexample to this statement:\\

\begin{example}
Consider the unit circle $S^1 \subset \mathbb{R}^2$ endowed with the intrinsic metric. Then, we have that $[p,-p]=S^1$ for all $p \in S^1$.
\end{example}

\noindent \\ Keeping our remark in mind, it is important to note that endpoints defining congruent metric segments do exhibit a precise structural constraint. Namely, let us introduce the following definition:\\

\begin{definition}\label{def2.6}
Let $(X,d)$ be a metric space and let $x,y,x',y' \in X$. We say that  $\left\{ x, y\right\}$ and $ \left\{ x', y'\right\}$ are in a \emph{rectangular configuration} if the relations $d(x,y)=d(x',y')$, $d(x,y')=d(x',y)$, and $d(x,x')=d(y,y')$ hold.
\end{definition}

\vspace{0.3 cm}

\begin{figure}[H]
    \centering
   \includegraphics[width=6 cm]{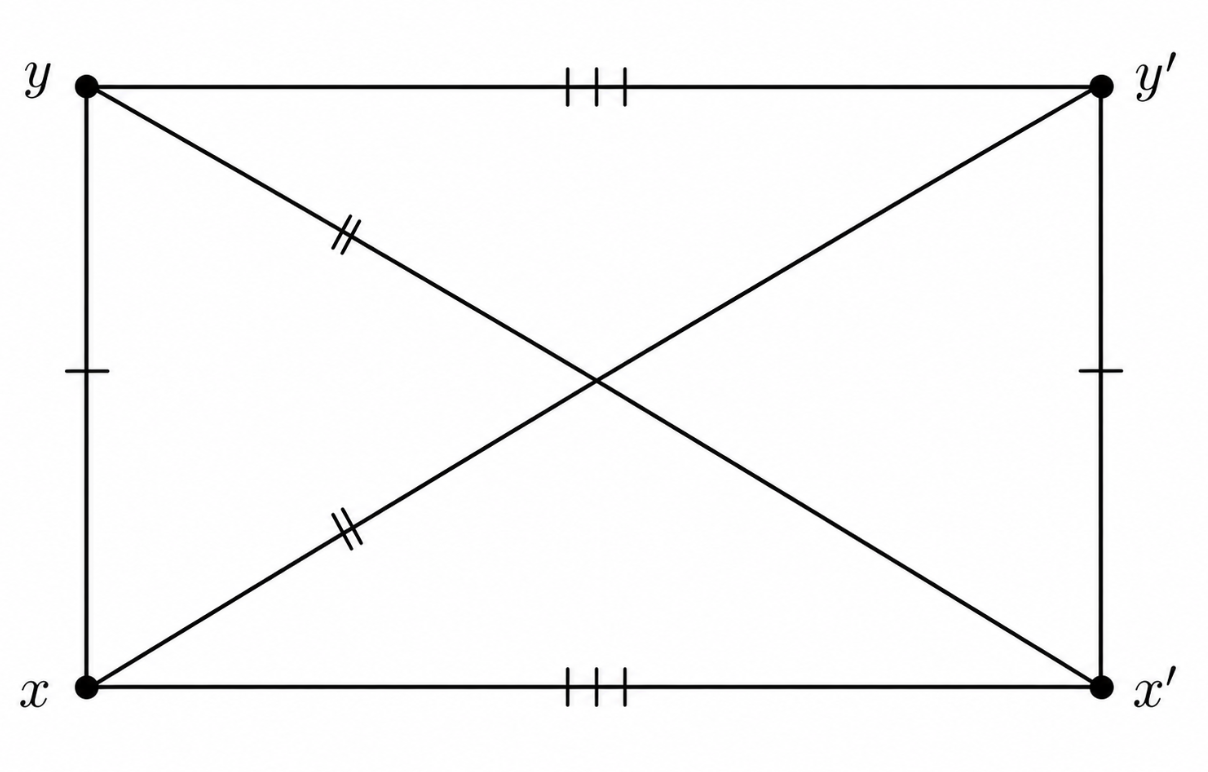} \caption{An illustration of the conditions of Definition~\ref{def2.6}.} \label{fig3}
    \end{figure}

\noindent \\ We have observed the following proposition:\\

\begin{proposition}\label{prop2.7}
Let $(X,d)$ be a metric space and assume $[x,y]=[x',y']$ for some $x,y,x',y' \in X$. Then, $ \left\{ x, y\right\}$ and $ \left\{ x', y'\right\}$ are in a rectangular configuration.
\end{proposition}

\vspace{0.5 cm}

\begin{proof}
By assumption, it holds that $\textup{diam}([x,y])=\textup{diam}([x',y'])$, so Proposition~\ref{prop2.2} implies $d(x,y)=d(x',y')$. Furthermore, from the relations $x \in [x',y']$ and $x' \in [x,y]$, we get\\

\begin{align*}
d(x',x)+d(x,y')&=d(x',y')\\
d(x,x')+d(x',y)&=d(x,y).
\end{align*}

\noindent \\However, since we already proved that $d(x,y)=d(x',y')$, it follows that $d(x,y')=d(x',y)$. Analogously, from the relation $x,y \in [x',y']$, we get\\

\begin{align*}
d(x',x)+d(x,y')&=d(x',y')\\
d(x',y)+d(y,y')&=d(x,y).
\end{align*}

\noindent \\However, since we have already proved $d(x,y)=d(x',y')$ and $d(x,y')=d(x',y)$, it follows that $d(x,x')=d(y,y')$.
\end{proof}

\noindent \\ Informally speaking, Proposition~\ref{prop2.7} gives more context to the observation that metric segments in $\mathbb{R}^2$ endowed with the $\ell^1$ and $\ell^{\infty}$ metrics are rectangles (see Figure~\ref{fig1}).\\

\noindent Furthermore, the property $(iii)$ in Proposition~\ref{prop2.4} alludes to the possibility that metric segments are themselves convex sets. However, this is not true in general, as can easily be seen from the following example:\\

\begin{example}
Consider the connected graph $G$ represented by the diagram below:\\

\begin{center}
\includegraphics[width=6cm]{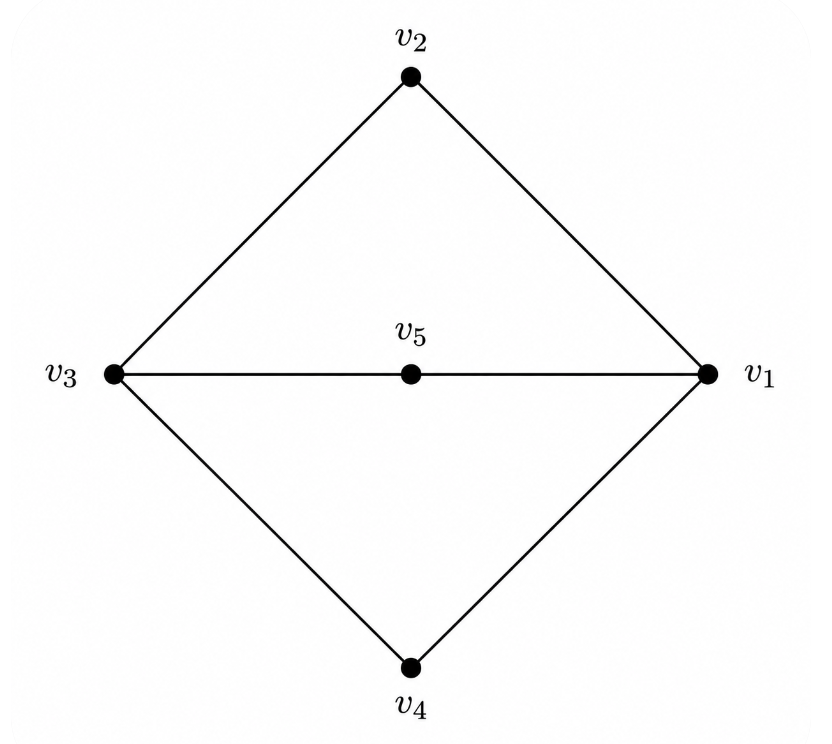}
\end{center}

\newpage

\captionof{figure}{A connected graph on 5 vertices as a counterexample to the convexity of metric segments.}
\label{fig4}

\noindent \\endowed with the usual geodesic metric. Observe that we have $[v_2, v_4]=G-\left\{ v_5 \right\}$. However, we also have $[v_1,v_3]=G$. Thus, $v_1,v_3 \in [v_2, v_4]$, but $[v_1, v_3] \nsubseteq [v_2,v_4]$. Hence, $[v_2, v_4]$ is non-convex.\footnote{In a precise sense, $G$ is the smallest connected graph that is a counterexample to the convexity of metric segments. Namely, all connected graphs with $\leq 4$ vertices have convex metric segments. Additionally, $G$ is the smallest connected graph on 5 vertices with a non-convex metric segment by number of edges.}
\end{example}

\vspace{0.3 cm}

Our discussion so far prompts us to inquire about whether it is possible to classify the different shapes metric segments can take on in metric spaces. We contribute to this problem by first noting the following proposition:\\

\begin{proposition}\label{prop2.9}
Let $X$ be a set and let $K \subseteq X$ be a subset of $X$ with $\left | K\right |>1$. Then, given any distinct $x,y$ in $K$, there exists a metric $d$ on X so that $[x,y]=K$.
\end{proposition}

\vspace{0.3 cm}

\begin{proof}
We define a function $d: X^2 \rightarrow \mathbb{R}_{\geq 0}$ so that $\forall p \in X: d(p,p)=0$ and for all distinct $p,q \in X$, we have\\

\begin{itemize}
     \item
     $d(p,q)=1$ if one of $p,q$ is $x$ or $y$, and the other is in $K-\left\{ x, y\right\}$;
     \item
     $d(p,q)=2$ if 
     $\left\{ p,q\right\}=\left\{ x,y \right\}$ or $p,q \in K-\left\{ x, y\right\}$;
     \item
     $d(p,q)=3$ if at least one of $p,q$ is not in $K$.
\end{itemize}

\noindent \\ Clearly, $d$ is positive-definite and symmetric. Furthermore, assume $p,q,r \in X$ are distinct points. We check the triangle inequality $d(p,r)+d(r,q) \geq d(p,q)$. The only non-trivial case to check is when $d(p,q)=3$ i.e., when at least one of $p,q$ is not in $K$. Without loss of generality, we can assume $p \notin K$. The only way for the triangle inequality to be violated is if we had $d(p,r)=d(r,q)=1$. However, $d(p,r)=1$ would imply $p \in K$, which is a contradiction. Thus, $d$ satisfies the triangle inequality and is hence a metric on $X$. Finally, one can easily verify that $[x,y]=K$.
\end{proof}

\vspace{0.3 cm}

\noindent Proposition~\ref{prop2.9} implies that metric segments can essentially take arbitrary shapes in metric spaces. In fact, we can prove a stronger result:\\

\begin{theorem}\label{thm2.10}
Let $(X,d)$ be a metric space and let $K \subseteq X$ be a closed subset of $X$ with $\left | K\right |>1$. Then, given any distinct $x,y$ in $K$, there exists a metric $\tilde{d}$ bounded and topologically equivalent to $d$ so that $[x,y]_{\tilde{d}}=K$.
\end{theorem}

\vspace{0.3 cm}

\begin{proof}
Set $\hat{d}=\min\left\{ d, 1\right\}$ and on $K$, set\\

\begin{equation}\label{2.1}
f
(p)=\frac{\hat{d}(x,p)}{\hat{d}(x,p)+\hat{d}(p,y)}
\end{equation}\\

\noindent for all $p \in K$. Then, $f$ defines a continuous function $f: K \rightarrow [0,1]$. As $K$ is closed, by the Tietze extension theorem, there is a continuous extension $F: X \rightarrow [0,1]$ of $f$ to $X$.  Finally, we also set $g(p)=\hat{d}(p,K)$ for all $p \in X$. Then, $g$ defines a continuous function $g: X \rightarrow [0, 1]$. Now, we can define $\tilde{d}$ so that\\

$$\tilde{d}(p,q)=\max\left\{\frac{1}{2} \hat{d}(p,q), \left | F(p)-F(q)\right |+ \left | g(p)-g(q)\right | \right\}$$\\

\noindent for all $p,q \in X$. Clearly, $\tilde{d}$ is bounded. We proceed by proving the following claims:\\

\noindent \textbf{Claim 1:} $\tilde{d}$ is a metric topologically equivalent to $d$.\\

\noindent The fact that $\tilde{d}$ is a metric follows from the fact that $\hat{d}$ is a metric and the observation that the expression $\left | F(p)-F(q)\right |+\left | g(p)-g(q)\right |$ defines a pseudometric on $X$. Furthermore, as $\hat{d}$ is topologically equivalent to $d$, it suffices to establish the topological equivalence of $\tilde{d}$ and $\hat{d}$. Thus, let $p\in X$ and let $(p_n)$ be a sequence of points in $X$. Definitionally, we have the relation $\frac{1}{2}\hat{d}\leq \tilde{d}$. Hence, the convergence $p_n \rightarrow p$ in $(X, \tilde{d})$ implies the convergence $p_n \rightarrow p$ in $(X,\hat{d})$. Conversely, assume $p_n \rightarrow p$ in $(X,\hat{d})$. Then, $\hat{d}(p_n,p) \rightarrow 0$ by definition. Additionally, we have $\left | F(p_n)-F(p)\right |\rightarrow 0$ and $\left | g(p_n)-g(p)\right |\rightarrow 0$ by continuity of $F$ and $g$ in $(X, \hat{d})$, respectively. Hence, $\tilde{d}(p_n,p) \rightarrow 0$, and $p_n \rightarrow p$ in $(X, \tilde{d}$). Indeed, $\tilde{d}$ and $\hat{d}$ are topologically equivalent. \\

\noindent \textbf{Claim 2:} $[x,y]_{\tilde{d}}=K$.\\

\noindent It is very beneficial to observe the relations $F(x)=0$, $F(y)=1$, and $\forall p \in X: (g(p)=0 \Leftrightarrow p \in K)$ (the last relation holds as $K$ is also closed in $(X, \hat{d}$)). Now, we easily calculate $\tilde{d}(x,y)=1$ and obtain the expressions since\\
\begin{align}
\tilde{d}(x,p)
&=\max\left\{ \frac{1}{2}\hat{d}(x,p), F(p)+g(p)\right\}
\label{2.2}\\
\tilde{d}(p,y)
&=\max\left\{ \frac{1}{2}\hat{d}(p,y), 1-F(p)+g(p)\right\}
\label{2.3}
\end{align}

\noindent \\ for all $p \in X$. Hence, we always have that\\

$$\tilde{d}(x,p)+\tilde{d}(p,y)\geq 1+2g(p)$$\\

\noindent When $p \notin K$, this ensures that $p \notin [x,y]_{\tilde{d}}$, so $[x,y]_{\tilde{d}} \subseteq K$. Finally, we show that $K \subseteq [x,y]_{\tilde{d}}$. Namely, assume $p \in K$. Then, \hyperref[2.2]{(2.2)}  and \hyperref[2.3]{(2.3)} give\\

\begin{align*}
\tilde{d}(x,p)&=\max\left\{ \frac{1}{2}\hat{d}(x,p), F(p)\right\}\\
\tilde{d}(p,y)&=\max\left\{ \frac{1}{2}\hat{d}(p,y), 1-F(p)\right\}.
\end{align*}\\

\noindent Furthermore, the bound $\hat{d} \leq 1$ gives the relations

\begin{align*}
F(p)=f(p) &\geq \frac{\hat{d}(x,p)}{2} \\
1-F(p)=1-f(p) &\geq \frac{\hat{d}(p,y)}{2}
\end{align*}

\noindent \\ by \hyperref[2.1]{(2.1)}. Thus, it follows that we have $\tilde{d}(x,p)=F(p)$ and $\tilde{d}(p,y)=1-F(p)$. Hence, $p \in [x,y]_{\tilde{d}}$, so $K \subseteq [x,y]_{\tilde{d}}$. This finishes the proof.
\end{proof}

\vspace{0.3 cm}

\section{Metrically, Menger, and Strictly Convex Metric Spaces} \label{sec3}

In this section, we want to draw particular attention to three types of metric spaces crucial to the study of metric segments and $d$-convexity, namely, metrically, Menger, and strictly convex spaces. We first provide the definitions of these spaces below:\\

\begin{definition}\label{def3.1}
A metric space $(X,d)$ is said to be\\

\begin{enumerate}
\itemsep0.3em
\item[(a)] \emph{metrically convex} if $\left [ x,y \right ]-\left\{ x,y\right\} \neq \emptyset$ for all distinct $x,y \in X, x 
\neq y$;
\item[(b)]  \emph{Menger convex} if for all distinct $x,y \in X, x \neq y$ and $\alpha \in [0,1]$ there exists $z_{\alpha} \in X$ with $d(x,z_{\alpha})=\alpha d(x,y)$ and $d(z_{\alpha},y)=(1-\alpha) d(x,y)$;
\item[(c)]  \emph{strictly convex} if for all distinct $x,y \in X, x \neq y$ and $\alpha \in [0,1]$ there exists a \emph{unique} $z_{\alpha} \in X$ with $d(x,z_{\alpha})=\alpha d(x,y)$ and $d(z_{\alpha},y)=(1-\alpha) d(x,y)$.
\end{enumerate}
\end{definition}

\noindent \\Our terminology is borrowed from Bula \cite{bula1999strictly}. The notions in Definition~\ref{def3.1} are usually studied in the context of \emph{geodesics} in metric spaces i.e., paths that are also isometric embeddings (this will be discussed in more detail in Section~\ref{sec4}). However, in this section, we will take a different approach and study these spaces purely in the context of $d$-convexity. Clearly, one has the chain of implications $\textup{strictly convex} \Rightarrow \textup{Menger convex} \Rightarrow \textup{metrically convex}$. Moreover, for complete metric spaces, metric and Menger convexity actually coincide due to a theorem of Menger (see \cite[Erste Untersuchung, §5]{menger1928untersuchungen}). Certainly, for incomplete metric spaces, we cannot expect such an equivalence. For instance, the rationals $\mathbb{Q}$ under the Euclidean metric are metrically convex, but not Menger convex. Furthermore, it is worthwhile to note that the definition of a strictly convex metric space is not ambiguous from the perspective of the theory of normed spaces, as Proposition~\ref{prop1.1} implies that a normed vector space is strictly convex (in the usual sense) if and only if it is strictly convex in the sense of Definition~\ref{def3.1}. Thus, the notion of  strictly convex metric spaces is an adequate generalization of the strict convexity of normed vector spaces.\\

\noindent Arguably, the easiest way to ensure the Menger convexity of a metric space is to require  metric segments to be connected. Indeed, under this assumption on $(X,d)$, for all $x,y \in X, x \neq y$, we can construct a map $\phi_{xy}: [x,y] \rightarrow [0, d(x,y)]$ defined by $\phi_{xy}: z \in [x,y] \mapsto d(x,z)$, which is then surjective by the intermediate value theorem. Consequently, for all $\alpha \in [0,1]$, a choice of $z_\alpha \in \phi_{xy}^{-1}(\left\{ \alpha d(x,y)\right\})$ gives the associated point $z_\alpha$. Interestingly, the converse need not hold. In fact, a Menger convex metric space need not even be connected itself, as can be seen from the following example:\\

\begin{example}
Let $c_{00}$ be the space of real-valued sequences with finite support endowed with the $\ell^{\infty}$ norm. Also, define a continuous linear functional $\varphi$ on $c_{00}$ given by the rule $\varphi: a \mapsto \sum 2^{-n}a_n$. Then, consider $X=\varphi^{-1}(\mathbb{R}-\left\{ 0\right\})$ as a metric subspace of $c_{00}$. First, we claim that $X$ is Menger convex. Thus, let $x,y, x \neq y$ be a pair of distinct sequences in $X$ and let $\alpha \in [0,1]$. Additionally, set $u=(1-\alpha)x+\alpha y$. If $\varphi(u)\neq 0$, then $u \in X$, so $z_{\alpha}=u$ gives the associated point to $\alpha$. However, if $\varphi(u)=0$, choose $i \in \mathbb{N}$ outside of $\textup{supp}(x) \cup \textup{supp}(y)$ and $\varepsilon>0$ with

$$\varepsilon \leq \min\left\{ \alpha, 1-\alpha \right\} \left\| x-y\right\|_{\infty}.$$

\noindent \\ Setting $z_{\alpha}=u+\varepsilon e_i$, note that we have $\varphi(u+\varepsilon e_i)=2^{-i}\varepsilon$, so $z_\alpha \in X$. Finally, due to the bound of $\varepsilon$, we retain the relations\\

\begin{align*}
\left\| x-z_{\alpha}\right\|&=\alpha\left\| x-y\right\| \\
\left\| z_{\alpha}-y\right\|&=(1-\alpha)\left\| x-y\right\|,
\end{align*}\\

\noindent making $z_{\alpha}$ an associated point to $\alpha$. Hence, $X$ is Menger convex. However, $X$ is not connected, as the subsets $\varphi^{-1}(\mathbb{R}^{+})$ and $\varphi^{-1}(\mathbb{R}^{-})$ give a separation of $X$. 
\end{example}

\noindent \\ That being said, it is important to note that metric segments are indeed connected when $X$ is a complete Menger convex metric space as a consequence of another classical theorem of Menger (see \cite[Erste Untersuchung, §6]{menger1928untersuchungen}).\\

On the other hand, comparing the definitions of Menger convex and strictly convex metric spaces in Definition~\ref{def3.1}, we see that Menger convexity only implies the \emph{existence} of the points $z_{\alpha}$ for all $\alpha \in [0,1]$, while strict convexity also ensures \emph{uniqueness} of these points. For instance, $\mathbb{R}^2$ endowed with the $\ell^1$ and $\ell^{\infty}$ norms is a Menger convex, but not a strictly convex space. Hence, it is not surprising that it is usually harder to establish the strict convexity of metric spaces in comparison to Menger convexity. However, we have observed the following theorem, which may provide an easier method to establishing the strict convexity of metric spaces:\\

\begin{theorem}\label{thm3.3}
Let $(X,d)$ be a Menger convex metric space. Then, $X$ is strictly convex if and only if one of the following equivalent conditions is additionally satisfied:\\

\begin{enumerate}
\itemsep0.3em
\item[(i)] For all $x,y \in X$ and $p,q \in [x,y]$, we have that $p \in [x,q]$ whenever $d(x,p) \leq d(x,q)$.
\item[(ii)] For all $x, y \in X$ and $p,q \in [x,y]$, we have that $p \in [x,q]$ or $q \in [x,p]$.
\item[(iii)] For all $x, y \in X$, the set $\left\{ [x,p] : p \in [x,y] \right\}$ forms a chain under inclusion.
\item[(iv)] Every metric segment in $X$ is isometrically embeddable into $\mathbb{R}$.
\end{enumerate}
\end{theorem}

\vspace{0.3 cm}

\begin{proof}
We start by showing the following:\\

\noindent \textbf{Claim 1:} The conditions $(i)-(iv)$ are equivalent.\\

\noindent $(i) \Rightarrow (ii)$ Trivial.\\

\noindent $(ii) \Rightarrow (iii)$ Let $x,y \in X$ and $p,q \in [x,y]$. By $(ii)$, we have two cases. If $p \in [x,q]$, then Proposition~\ref{prop2.4} under $(iii)$ gives $[x,p] \subseteq [x,q]$. Similarly, if $q \in [x,p]$, we get $[x,q] \subseteq [x,p]$ by symmetry. Hence, $\left\{ [x,p] : p \in [x,y] \right\}$ is a chain under inclusion.\\

\noindent $(iii) \Rightarrow (iv)$ Let $x,y \in X$. We claim that the map $\phi_{xy}: [x,y] \rightarrow [0, d(x,y)]$ as defined previously is an isometric embedding. Thus, let $p,q \in [x,y]$. By $(iii)$, we can assume without loss of generality that $[x,q] \subseteq [x,p]$. Then, in particular, we have $q \in [x,p]$, so it follows that\\

\begin{align*}
&d(x,q)+d(q,p)=d(x,p)\\
\Leftrightarrow & d(x,p)-d(x,q)=d(p,q)\\
\Rightarrow & \left | \phi_{xy}(p)-\phi_{xy}(q)\right |=d(p,q).
\end{align*}

\noindent \\ Hence, $\phi_{xy}$ is an isometric embedding of $[x,y]$ into $\mathbb{R}$.\\

\noindent $(iv)\Rightarrow (i)$ Let $x,y \in X$, $p,q \in [x,y]$, and assume that $d(x,p) \leq d(x,q)$. Furthermore, choose an isometric embedding $f: [x,y] \hookrightarrow \mathbb{R}$ of $[x,y]$ into $\mathbb{R}$. Then, it is not difficult to observe that our assumptions translate into the statements $f(p),f(q) \in [f(x),f(y)]_{\mathbb{R}}$ and $\left | f(x)-f(p)\right | \leq \left | f(x)-f(q)\right | $ in $\mathbb{R}$. Now, it clearly follows that $f(p) \in [f(x), f(q)]_{\mathbb{R}}$, which is equivalent to the statement $p \in [x,q]$ in $X$.\\

Having established the equivalence of $(i)-(iv)$, we now, without loss of generality, assume that $(iii)$ holds. Then, as we have already proved, $\phi_{xy}$ is an isometric embedding for all $x,y \in X$. Hence, in particular, $\phi_{xy}$ is injective. However, as $\phi_{xy}$ is additionally surjective by the assumption that $X$ is Menger convex, we conclude that $\phi_{xy}$ is a bijection. Consequently, $X$ is strictly convex.\\

Conversely, assume that $X$ is strictly convex. We will, without loss of generality, conclude $(i)$. Thus, let $x,y \in X$. Without loss of generality, we can assume $x \neq y$. Furthermore, let $p,q \in [x,y]$ and assume $d(x,p) \leq d(x,q)$. Without loss of generality, we can additionally assume $p,q \notin \left\{ x,y\right\}$. Now, set

$$\alpha=\frac{d(x,p)}{d(x,q)}.$$

\noindent \\ Then, under our assumptions, we have that $\alpha \in (0,1]$. Hence, by strict convexity, there exists a unique $z_{\alpha} \in [x,q]$ satisfying $d(x,z_{\alpha})=\alpha d(x,q)$ and $d(z_{\alpha},q)=(1-\alpha)d(x,q)$. It suffices to prove that  $z_{\alpha}=p$. By Proposition~\ref{prop2.4} under $(iii)$, we have $[x,q] \subseteq [x,y]$, hence $z_{\alpha} \in [x,y]$. Additionally, we have by definition that\\

$$\frac{d(x,z_{\alpha})}{d(x,y)}=\frac{\alpha d(x,q)}{d(x,y)}=\frac{d(x,p)}{d(x,y)}.$$\\

\noindent However, as $p \in [x,y]$, the uniqueness of $z_{\alpha}$ necessitates that $p=z_{\alpha}$. This finishes the proof.
\end{proof}

\vspace{0.3 cm}

\noindent We note that  our Theorem~\ref{thm3.3} intersects with some  results that are already known in $d$-convexity. For instance, Menger observed the following:\\

\begin{theorem}[Menger, {\cite[Zweite Untersuchung, \S 10]{menger1928untersuchungen}}]\label{thm3.4}
Let $(X,d)$ be a metric space with $|X|>4$. Then, $X$ isometrically embeddable into $\mathbb{R}$ if and only if at least one of the statements $x \in [y,z]$, $y \in [x,z]$, and $z \in [x,y]$ is satisfied for all $x,y,z \in X$.
\end{theorem}

\noindent \\ The  bi-implication $(ii) \Leftrightarrow (iv)$, as established in our Theorem \ref{thm3.3}, closely resembles the statement of Theorem~
\ref{thm3.4} in the case when  metric segments are considered as separate metric subspaces. However, the proof is not immediate. One can indeed show $(ii) \Leftrightarrow (iv)$ using Theorem~\ref{thm3.4}, but we find this approach to be much more tedious than ours and, consequently, omit it. Additionally, we propose the following terminology:\\

\begin{definition}
A metric space is said to be \emph{directed} if it satisfies any of the equivalent conditions $(i)-(iv)$ of Theorem~\ref{thm3.3}.
\end{definition}

\vspace{0.3em}

\begin{figure}[H]
    \centering   \includegraphics[width=9 cm]{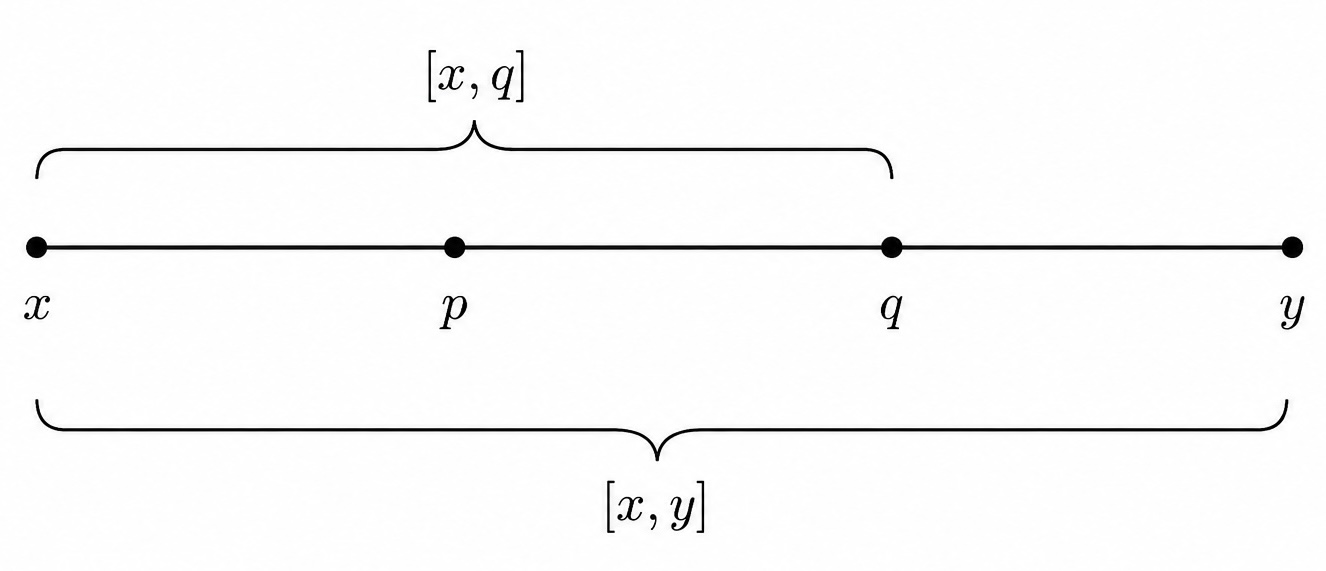}
   \caption{An illustration of condition $(i)$ of Theorem~\ref{thm3.3} as a criterion of directedness.} \label{fig5}
    \end{figure}

\vspace{0.2 cm}

\noindent Definitionally, directed spaces admit metric segments that are isometrically embeddable into $\mathbb{R}$. Furthermore, in $\mathbb{R}$, line segments are convex and have unique endpoints. We note that these properties transfer to general directed spaces, as illustrated by the following proposition:

\vspace{0.2 cm}

\begin{proposition}\label{prop3.6}
In a directed space, metric segments are convex and have unique endpoints.
\end{proposition}

\vspace{0.1 cm}

\begin{proof}
Assume $(X,d)$ is directed and let $x,y \in X$. Additionally, choose $p,q \in [x,y]$. By directedness, we know that $p \in [x,q]$ or $q \in [x,p]$. Without loss of generality, we can assume that $p \in [x,q]$. Applying Proposition~\ref{prop2.4}
under $(iii)$ gives the sequence of inclusions $[p,q] \subseteq [x,q] \subseteq [x,y]$. Hence, $[p,q] \subseteq [x,y]$, and $[x,y]$ is convex.\\

\noindent Furthermore, to establish uniqueness of endpoints, let $x',y' \in X$ with $[x',y']=[x,y]=S$. As $X$ is directed, there is an isometric embedding $f: S \rightarrow \mathbb{R}$. Clearly, we have that $\textup{diam}(f(S))=\textup{diam}(S)$. Additionally,  Proposition~\ref{prop2.2} under $(ii)$ implies\\

$$\textup{diam}(S)=d(x,y)=d(x',y')$$

\noindent\\ and the assumption that $f$ is an isometry gives\\

\begin{align*}
\left | f(x)-f(y)\right |&=d(x,y)\\
\left | f(x')-f(y')\right |&=d(x',y').\\
\end{align*}

\vspace{0.3 cm}

\noindent Consequently, $\left\{ f(x),f(y)\right\}$ and $\left\{ f(x'),f(y')\right\}$ are pairs realizing the diameter of $f(S)$. However, as $f(S)\subset \mathbb{R}$ and diameter realizing pairs in $\mathbb{R}$ are unique, this necessitates $\left\{ f(x),f(y)\right\}=\left\{ f(x'),f(y')\right\}$. Finally, since $f$ is an injection, we conclude that $\left\{ x,y\right\}=\left\{ x',y'\right\}$.
\end{proof}

\noindent \\ Theorem~\ref{thm3.3} and Proposition~\ref{prop3.6} together imply the following expected corollary:\\

\begin{corollary}
In a strictly convex space, metric segments are convex and have unique endpoints.
\end{corollary}

\vspace{0.2 cm}

\noindent That being said, unsurprisingly, the topological properties of line segments in $\mathbb{R}$ such as compactness and connectedness do not transfer to general directed spaces. For instance, the rationals $\mathbb{Q}$ under the Euclidean metric are directed, but nontrivial metric segments  in $\mathbb{Q}$ are not compact (see Example~\ref{ex2.3}) and are totally disconnected.

\vspace{0.3 cm}

\section{Convex Sets in Metric Spaces} \label{sec4}
As announced, we will begin this section by outlining the most common notions of convexity of sets in metric spaces other than $d$-convexity and describing their relationships to $d$-convexity. We have organized these notions into four categories - linear, geodesic, $W$-convexity, and a combined metric and Menger convexity category. In the subsequent subsections, $(X,d)$ will denote an arbitrary metric space. Additionally, the reader should keep in mind the following remark, the proof of which should be apparent:\\

\begin{remark}\label{rem4.1}
Let $I: X^2 \rightarrow \mathcal{P}(X)$ be an arbitrary map. Then, the collection

\vspace{0.2 cm}

$$\mathcal{C}_I=\left\{ C \subseteq X \mid \forall x,y \in C : I(x,y) \subseteq C\right\}$$

\noindent \\ forms a convexity on $X$. We say that the convexity $\mathcal{C}_I$ is \emph{induced} by $I$.
\end{remark}

\noindent \\ For instance, setting $I(x,y)=[x,y]$ for all $x,y \in X$, one obtains the usual convexity $\mathcal{C}_d$ induced by the metric $d$.

\medskip

\subsection{Linear Convexity}\label{sub4.1}

\leavevmode

\noindent Oftentimes, it is the case that $X$ is a real vector space and $d$ is compatible with the linear structure on $X$ (i.e., $d$ is translation invariant and absolutely homogeneous). This is so, e.g., when $X$ is a normed space, and $d$ the induced metric. In such scenarios, it is sensible to consider the convexity of  linearly convex subsets of $X$, which we will denote by $\mathcal{C}^L_X$. When $X$ is strictly convex, one easily sees that metric segments agree with line segments 
in $X$. Consequently, linear and $d$-convexity agree on $X$ i.e., we have the relation $\mathcal{C}^L_X=\mathcal{C}_d$.  This agreement is a generalization of the case when $X$ is a normed space that was discussed in Section~\ref{sec1}. Furthermore, if $X$ is not necessarily strictly convex, one still observes the inclusion

\vspace{0.2 cm}

$$\left\{ (1-t)x+ty: t \in \left [ 0,1 \right ]\right\} \subseteq [x,y]$$

\noindent \\ for all $x,y \in X$. Consequently, $d$-convexity implies linear convexity in $X$ i.e., $\mathcal{C}^L_X$ is finer than $\mathcal{C}_d$. To see this on a particular example, consider the segment $S=[(0,0),(1,1)]_{d_2}$ under the Euclidean norm on $\mathbb{R}^2$. Then, $S$ is certainly linearly convex, but it is not convex under the $\ell^1$ norm on $\mathbb{R}^2$.

\medskip

\subsection{Geodesic Convexity}\label{sub4.2}
\leavevmode

\noindent Geodesic convexity provides a geometric notion of convexity of sets in metric spaces, relying on geodesics to define convex sets, as opposed to metric segments. We start this subsection by first formally defining the notion of a geodesic in a metric space:

\vspace{0.2 cm}

\begin{definition}
Given $x,y \in X$, a \emph{geodesic} joining $x,y$ is a path $\gamma: [0, d(x,y)] \rightarrow X$ from $x$ to $y$ that is also an isometric embedding.
\end{definition}

\vspace{0.2 cm}

\noindent (A useful intuition for the reader is that geodesics are paths that minimize length). We will assume that $X$ is \emph{geodesic}, so that any pair of points in $X$ can be joined by a geodesic. Examples of geodesic metric spaces include complete metrically convex spaces, as a consequence of the aforementioned theorem of Menger (see \cite[Erste Untersuchung, §6]{menger1928untersuchungen}), and connected graphs. Of course, geodesics between a pair of points need not be unique (the reader can convince themselves in this, e.g., by considering the geodesics of connected graphs). Thus, authors typically distinguish between the following two notions of geodesic convexity:\\

\begin{definition}
A subset $C \subseteq X$ is said to be

\vspace{0.2 cm}

\begin{enumerate}
\itemsep0.3em
\item[(a)] \emph{weakly convex} if every pair of points in $C$ can be joined by a geodesic whose image is contained in $C$ (or, equivalently, if $C$ is geodesic as a subspace of $X$);
\item[(b)] \emph{totally convex} if for every pair $x,y \in C$ of points in $C$, the image of \emph{all} geodesics joining $x,y$ is contained in $C$.
\end{enumerate}
\end{definition}

\vspace{0.2 cm}

\noindent (see e.g., \cite[§2.D]{alexander2024alexandrov} for a reference). Immediately, one can observe an apparent problem with weak convexity. Specifically, as we have discussed in Section~\ref{sec1}, notions of convexity of sets in numerous structures are aptly modeled under axiomatic convexity. However, while the collection of weakly convex subsets of $X$ satisfies the axioms \hyperref[c1]{C1} and \hyperref[c3]{C3}, it need not satisfy the axiom \hyperref[c2]{C2} (closure under arbitrary intersection). Consequently, weakly convex subsets of $X$ need not form a convexity, making weak convexity a poor notion of convexity of sets. 
An instance of the failure of the axiom \hyperref[c2]{C2} for weakly convex sets is given by the following example:\\

\begin{example}\label{ex4.4}
Consider the unit circle $S^1 \subset \mathbb{R}^2$ endowed with the intrinsic metric, which is clearly geodesic, and choose a point $p \in S^1$. Additionally, take $C_1$ and $C_2$ to be the two semicircular arcs bounded by the points $p,-p$ as illustrated by the following diagram:

\vspace{0.25cm}

\begin{center}
\includegraphics[width=5.5cm]{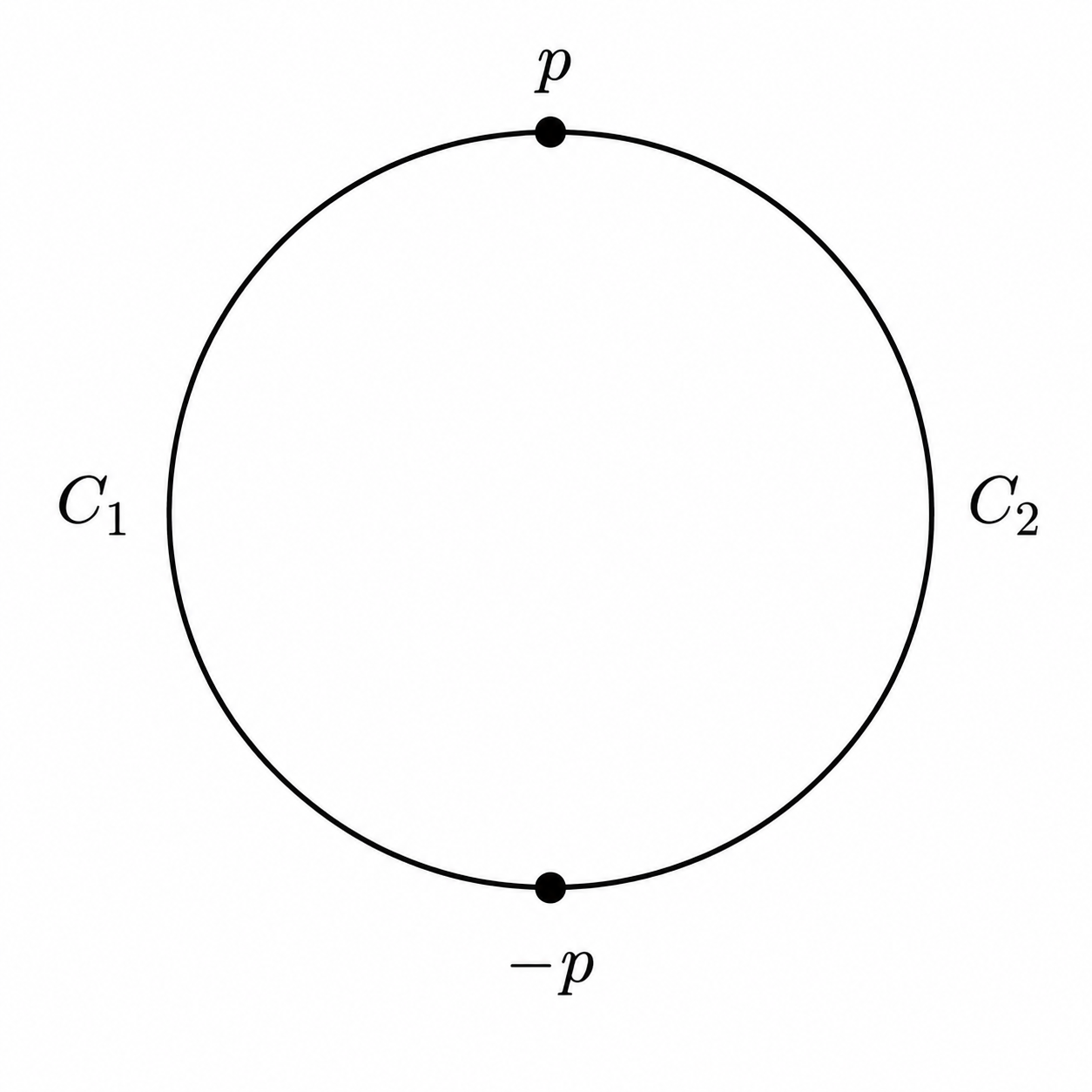}
\end{center}

\newpage

\captionof{figure}{Illustration of the failure of weakly convex subsets of the circle to be closed under intersection.}
\label{fig6}

\noindent \\ Then, $C_1$ and $C_2$ are weakly convex. However, $C_1 \cap C_2=\left\{ \pm p\right\}$, so $C_1 \cap C_2$ is not weakly convex.
\end{example}

\vspace{0.2 cm}

\noindent On the other hand, total convexity is much more well-behaved, as the collection of totally convex subsets of $X$ does form a convexity, which we will denote by $\mathcal{C}^G_d$. One can see this by setting

$$I(x,y)=\bigcup \left\{ \textup{im}(\gamma) : \gamma \ \textup{is a geodesic joining} \ x, y \right\}$$

\noindent \\ for all $x,y \in X$ and applying Remark~\ref{rem4.1}. However, there are still downsides to total convexity. The first, and most obvious downside is that total convexity is only defined for geodesic spaces, so it is not a universal notion of convexity of sets in metric spaces. The second, less obvious downside is that total convexity and $d$-convexity actually agree on geodesic spaces, making the notion redundant. Namely, with the map $I: X^2 \rightarrow \mathcal{P}(X)$ defined as above, the following proposition is well-known:\\

\begin{proposition}\label{prop4.5}
For all $x,y \in X$, it holds that $I(x,y)=[x,y]$.
\end{proposition}

\vspace{0.3 cm}

\begin{proof}
We first show the inclusion $I(x,y)\subseteq [x,y]$. Choose $p \in I(x,y)$. Then, by definition, $p=\gamma(t)$ for some geodesic $\gamma$ joining $x,y$ and $t \in [0, d(x,y)]$. Furthermore, using the defining properties of geodesics, we easily obtain

\begin{align*}
d(x,p)+d(p,y)&=d(\gamma(0), \gamma(t))+d(\gamma(t), \gamma(d(x,y)))\\
&=t+d(x,y)-t=d(x,y).
\end{align*}

\noindent \\ Hence, $p \in [x,y]$, so $I(x,y)\subseteq [x,y]$. Finally, we show $[x,y] \subseteq I(x,y)$. Namely, let $p \in [x,y]$. As $X$ is geodesic, there exist geodesics $\gamma_1$ and $\gamma_2$ joining $x,p$ and $p,y$, respectively. Now, we can consider their concatenation: \\

$$\gamma(t)=
\begin{cases}
\gamma_1(t), & t\in [0,d(x,p)],\\[4pt]
\gamma_2\bigl(t-d(x,p)\bigr), & t\in [d(x,p),d(x,y)]
\end{cases}$$\\

\noindent for all $t \in [0, d(x,y)]$. This expression clearly defines a path from $x$ to $y$. Furthermore, it is clear that the restrictions $\gamma\mid_{[0, d(x,p)]}$ and  $\gamma\mid_{[d(x,p), d(x,y)]}$ are isometric embeddings. Thus, to show that $\gamma$ is itself an isometric embedding, it suffices to establish that\\

$$d(\gamma(t_1), \gamma(t_2))=t_2-t_1$$

\noindent \\ for all $t_1 \in [0, d(x,p)]$ and $t_2 \in [d(x,p), d(x,y)]$. Using the triangle inequality, we have that\\

\begin{align*}
d(\gamma(t_1), \gamma(t_2)) &\leq 
d(\gamma(t_1),\gamma(d(x,p)))+d(\gamma(d(x,p)), \gamma(t_2))\\
&=d(x,p)-t_1+t_2-d(x,p)=t_2-t_1.
\end{align*}

\noindent \\ Analogously, we have that\\

\begin{align*}
d(x,y)&=d(\gamma(0), \gamma(d(x,y))\\
&\leq d(\gamma(0), \gamma(t_1))+d(\gamma(t_1), \gamma(t_2))+d(\gamma(t_2), \gamma(d(x,y))\\
&=t_1+d(\gamma(t_1), \gamma(t_2))+d(x,y)-t_2.
\end{align*}

\noindent \\ Hence, $t_2-t_1 \leq d(\gamma(t_1), \gamma(t_2))$. We conclude that $d(\gamma(t_1), \gamma(t_2))=t_2-t_1$, so $\gamma$ is indeed an isometric embedding. Consequently, $\gamma$ is a geodesic joining $x,y$. As $p \in \textup{im}(\gamma)$, we indeed have that $p \in I(x,y)$, so $[x,y] \subseteq I(x,y)$. This finishes the proof.
\end{proof}

\vspace{0.3 cm}

\noindent As an immediate corollary of Proposition~\ref{prop4.5}, we get that $\mathcal{C}^G_d=\mathcal{C}_d$, as claimed.

\medskip

\subsection{$W$-Convexity}\label{sub4.3}
\leavevmode

\noindent In his seminal 1970 paper (see \cite{takahashi1970convexity}), Takahashi generalized many results of the theory of linear convexity, such as the existence of fixed points of various classes of non-expansive maps, to metric spaces endowed with an adequately defined notion of taking "convex combinations" of points, which he called \emph{convex structures}. Formally, Takahashi introduced the following definition:

\vspace{0.2 cm}

\begin{definition}
A \emph{convex structure} on $X$ is a map $W: X^2 \times [0,1] \rightarrow X$ which satisfies the inequality

$$d(p,W(x,y;\alpha)) \leq \alpha d(x,p)+(1-\alpha) d(p,y)$$

\noindent \\ for all $p,x,y \in X$ and $\alpha \in [0,1]$.
\end{definition}

\noindent \\ The space $X$ is said to be \emph{Takahashi convex} provided that $X$ can be endowed with a convex structure $W$. In this subsection, we will assume that $X$ is Takahashi convex. Takahashi's framework proved to be quite useful due to the ability of convex structures to define linear convexity. Indeed, when $X$ is a real vector space and $d$ is compatible with the linear structure on $X$, one easily checks that the rule $W(x,y;\alpha)=\alpha x+(1-\alpha)y$ defines a convex structure on a given linearly convex subset $C \subseteq X$.\\

\noindent Given a convex structure $W$ on $X$, a subset $C \subseteq X$ is said to be \emph{$W$-convex} if $W(x,y; \alpha) \in C$ holds for all $x,y \in C$ and $\alpha \in [0,1]$. The collection of $W$-convex subsets of $X$ forms a convexity, which we will denote by $\mathcal{C}_W$. One can see this by setting\\

$$I(x,y)=\left\{ W(x,y; \alpha): \alpha \in [0,1]\right\}$$

\noindent \\ for all $x,y \in X$ and applying Remark~\ref{rem4.1}. Furthermore, Takahashi observed the following proposition:\\

\begin{proposition}[Takahashi, {\cite[\S 2]{takahashi1970convexity}}]\label{prop4.7}
Let $x,y \in X$ and $\alpha \in [0,1]$. Then, the following relations hold:

\begin{align*}
d(x, W(x,y;\alpha))&=(1-\alpha)d(x,y)\\
d(W(x,y;\alpha),y)&=\alpha d(x,y).
\end{align*}
\end{proposition}

\vspace{0.4 cm}

\noindent In other words, Propositon~\ref{prop4.7} claims that for a fixed $x,y \in X$ and $\alpha \in [0,1]$, $z_{\alpha}=W(x,y;\alpha)$ is the point associated to $1-\alpha$ in the sense of Menger convexity. It follows that $I(x,y) \subseteq [x,y]$ for all $x,y \in X$, with the map $I: X^2 \rightarrow \mathcal{P}(X)$ defined as above. Consequently, $d$-convexity implies $W$-convexity i.e., $\mathcal{C}_{W}$ is finer than $\mathcal{C}_d$ for any convex structure $W$. To see this on a particular example, consider again the segment $S=[(0,0),(1,1)]_{d_2}$ under the Euclidean norm on $\mathbb{R}^2$. Furthermore, the map $W$ taking convex combinations of points is a convex structure under the $\ell^{1}$ norm on $\mathbb{R}^2$. Now, $S$ is certainly $W$-convex, but it is not convex under the $\ell^1$ norm. That being said, when $X$ is strictly convex, $W$- and $d$-convexity agree on $X$ ($C_W=C_d$), as a consequence of Proposition~\ref{prop4.7}.\\

\noindent On another note, despite the ability of $W$-convexity to adequately model linear convexity, the notion is rather weak, as it is only defined for Takahashi convex metric spaces. Indeed, it is evident from Proposition~\ref{prop4.7} that a Takahashi convex metric space must, at the very least, be Menger convex.
 
\medskip

\subsection{Metric and Menger Convexity}\label{sub4.4}
\leavevmode

\noindent Assuming that $X$ is metrically convex, we can say that a subset $C \subseteq X$ is \emph{metrically convex} if $C$ is metrically convex for the induced metric on $C$. Analogously, we define what it means for $C$ to be \emph{Menger convex} when $X$ is a Menger convex space. For instance, taking $X=\mathbb{R}^n$, we observe that Menger convexity of subsets of $X$ corresponds precisely to linear convexity on $X$. However, as is the case with weak convexity, metric and Menger convexity are poor notions of convexity from the standpoint of axiomatic convexity theory. Specifically, while the collection of metrically (Menger) convex subsets of a metrically (Menger) convex space $X$ satisfies the axioms~\hyperref[c1]{C1} and ~\hyperref[c3]{C3}, it need not satisfy the axiom~\hyperref[c2]{C2} (closure under arbitrary intersection). Consequently, metrically (Menger) convex subsets of $X$ need not form a convexity. An instance of the failure of axiom~\hyperref[c2]{C2} for metrically (Menger) convex subsets is given by the same construction in Example~\ref{ex4.4}.

\vspace{0.3 cm}

Our discussion in Subsections~\ref{sub4.1}-\ref{sub4.4} offers compelling evidence that $d$-convexity is the optimal notion of convexity of sets in metric spaces that is well-suited to and useful in all general frameworks. In the final part of this section, we will explore the interaction of $d$-convexity and the inherent topological structure of a metric space. Unlike in normed spaces, it is well-known that the interior and closure of convex sets in metric spaces need not be convex. We will use the following terminology:\\

\begin{definition}
A metric space is said to have an \emph{interior-} (\emph{closure-}) \emph{stable convexity} if the interior (closure) of every convex set is convex.
\end{definition}

\noindent \\ We provide the following examples of a metric space with a convexity that is not interior-stable:\\

\begin{example}\label{ex4.9}
Consider the subset $X \subset \mathbb{R}^2$ represented by the diagram below:

\begin{figure}[H]
    \centering
    \includegraphics[width=6.5 cm]{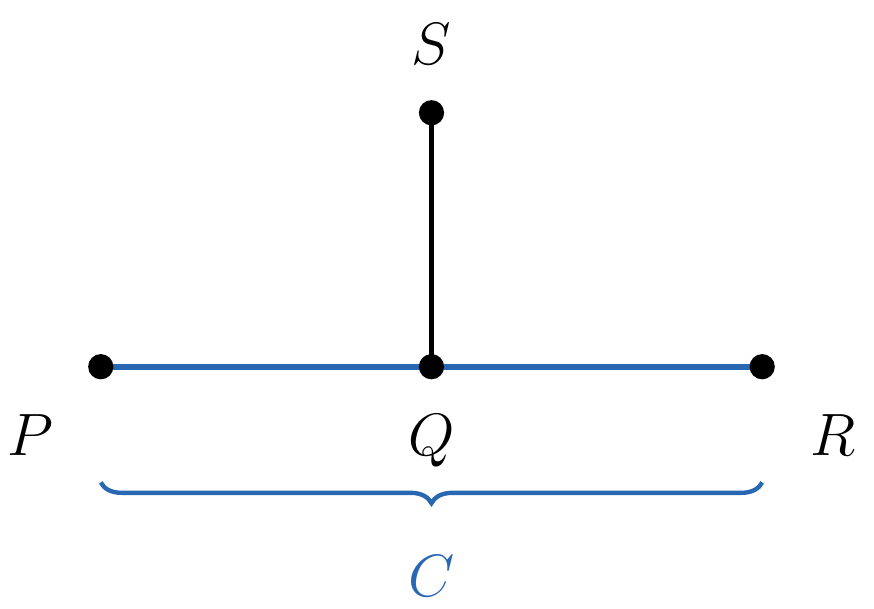}
   \caption{A counterexample to interior-stability.} \label{fig7}
    \end{figure}

\noindent \\ endowed with the intrinsic metric. Additionally, let $C$ denote the horizontal base segment of $X$. Then, $C$ is clearly convex as a subset of $X$. However, we have that $\textup{int}(C)=C-\left\{ P, Q, R\right\}$, which is non-convex. Indeed, choosing a pair of points in the interior of the segments $[P,Q]$ and $[Q,R]$, respectively, the segment joining these points must necessarily cross the point $Q$.

\end{example}

\noindent \\ Likewise, for closure-stability, we have the following counterexample:\\

\begin{example}\label{ex4.10}
Consider the unit circle $S^1 \subset \mathbb{R}^2$ endowed with the intrinsic metric and $C_1 \subset S^1$ as defined in Figure~\ref{fig6}. The subset $C=C_1-\left\{ \pm p \right\}$ is clearly convex. However, we have that $\overline{C}=C_1$, which is non-convex. Indeed, $\left\{ \pm p \right\} \subseteq C_1$, but $[p, -p]=S^1$.
\end{example}

\noindent \\ Fundamentally, the interior-stability of normed spaces is a consequence of the fact that the map taking points to their convex combination is open (for a fixed scalar). Exploiting this observation, we were able to characterize metric spaces with interior-stable convexities as follows:\\

\begin{theorem}\label{thm4.11}
A metric space $(X,d)$ has an interior-stable convexity if and only if \\ $\textup{conv}\left ( B(p, \varepsilon) \cup B(q, \delta)  \right )$ is an open set for all $p,q \in X$ and $\varepsilon, \delta>0$.
\end{theorem}

\vspace{0.3 em}

\begin{proof}
Assume that $\textup{conv}\left ( B(p, \varepsilon) \cup B(q, \delta)  \right )$ is open for all $p,q \in X$ and $\varepsilon, \delta>0$ and let $C \subseteq X$ be convex.  Without loss of generality, we can assume that $C$ has a non-empty interior. Additionally, choose $p,q \in \textup{int}(C)$. By definition, there exist $\varepsilon, \delta>0$ so that $B(p, \varepsilon), B(q, \delta) \subseteq C$. Thus, it follows that $B(p, \varepsilon) \cup B(q, \delta) \subseteq C$. Furthermore, as $C$ is convex, this implies that $\textup{conv}\left ( B(p, \varepsilon) \cup B(q, \delta)  \right ) \subseteq C$. Finally, as $\textup{conv}\left ( B(p, \varepsilon) \cup B(q, \delta)  \right )$ was assumed to be open, we get that

$$\textup{conv}\left ( B(p, \varepsilon) \cup B(q, \delta)  \right ) \subseteq \textup{int}(C).$$

\noindent \\ Now, the inclusion $[p,q] \subseteq \textup{conv}\left ( B(p, \varepsilon) \cup B(q, \delta)  \right )$ clearly implies $[p,q] \subseteq \textup{int}(C)$, so $\textup{int}(C)$ is convex, and $X$ has an interior-stable convexity.\\

Conversely, assume that $X$ has an interior-stable convexity. Let $p,q \in X$ and $\varepsilon, \delta>0$. Additionally, for the sake of brevity, set $S=B(p, \varepsilon) \cup B(q, \delta)$. It suffices to establish that $\textup{conv}\left ( S  \right ) \subseteq \textup{int}\left (\textup{conv}(S)\right )$. Clearly, we have the inclusion $S \subseteq \textup{conv}(S)$. Furthermore, as $S$ is open, it follows that $S \subseteq \textup{int}(\textup{conv}(S))$. Similarly, as $\textup{conv}(S)$ is convex, by assumption, we have that $\textup{int}(\textup{conv}(S))$ is also convex. Hence, $\textup{conv}(S) \subseteq \textup{int}(\textup{conv}(S))$, as desired.
\end{proof}

\vspace{0.3 cm}

\begin{figure}[H]
    \centering
    \includegraphics[width=9 cm]{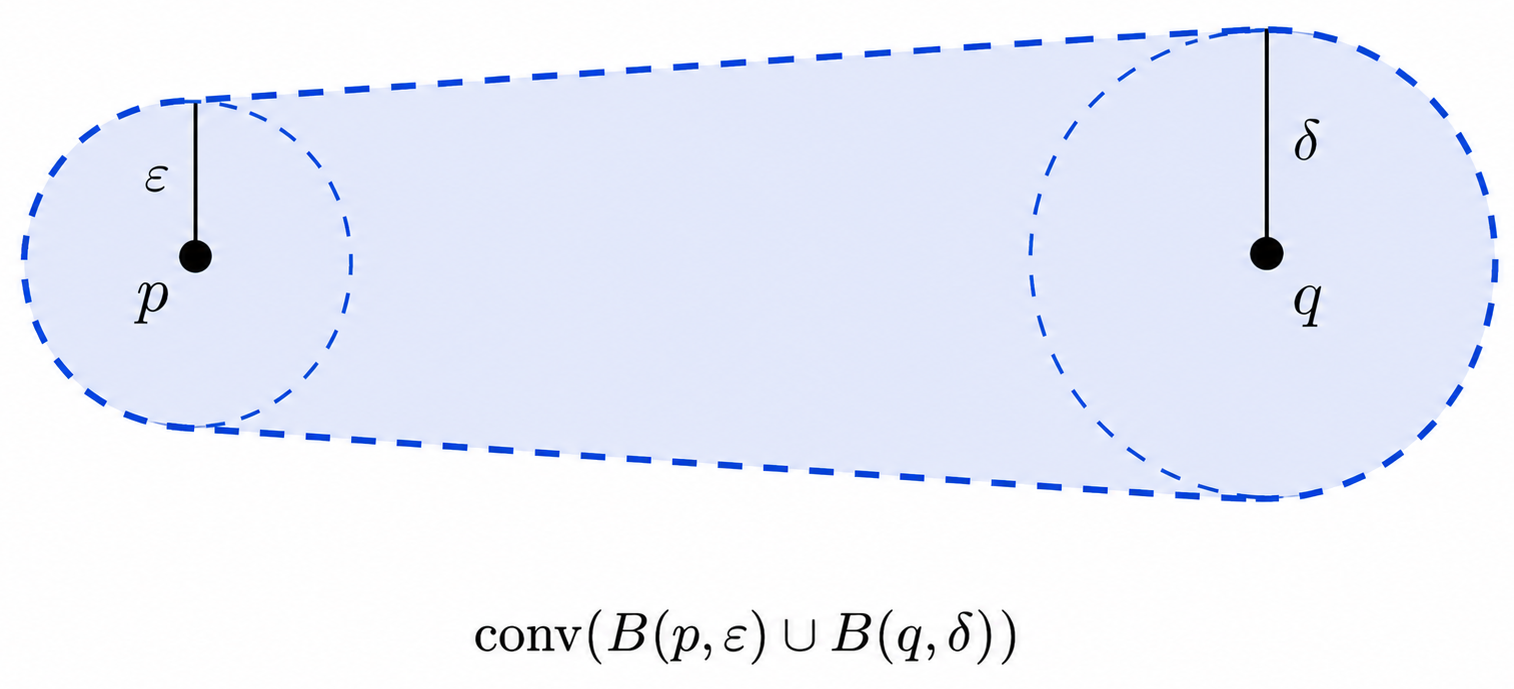}
   \caption{An illustration of the equivalent condition of Theorem~\ref{thm4.11}.} \label{fig8}
    \end{figure}

\noindent \\ We note that certain structural aspects of our proof of Theorem~\ref{thm4.11} are closely related to the recent work of Anderson et al. on interior-stability in the more general framework of topological betweenness structures (see \cite[\S 6]{anderson2021convexity}). On the other hand, the failure of metric spaces to admit closure-stable convexities in general (as is the case in Example~\ref{ex4.10}) is a consequence of the poor interaction of convergence and betweenness. Specifically, for an arbitrary metric space $X$, given points $x,y,z \in X$ with $z \in [x,y]$, the existence of convergent sequences $x_n \rightarrow x$ and $y_n \rightarrow y$ does not necessarily imply the existence of a sequence $z_n$ satisfying 
$z_n \in [x_n, y_n]$ for all $n \in \mathbb{N}$ and $z_n \rightarrow z$. Furthermore, this phenomenon is closely related to the branching of geodesics in metric spaces.\\

Apart from interior- and closure-stability, another aspect of the interaction of $d$-convexity and the topology of metric spaces that is of interest to study is the problem of classification of metrics that induce congruent metric segments or convex subsets. We introduce the following terminology:\\

\begin{definition}
A pair $d, \rho$ of metrics on a set $X$ is said to be\\

\begin{enumerate}
\itemsep0.3em
\item[(a)] \emph{segmentally equivalent} if 
the metrics induce congruent metric segments i.e., if $[x,y]_d=[x,y]_{\rho}$ holds for all $x,y \in X$;
\item[(b)] \emph{convexly equivalent} if the metrics induce congruent convex subsets i.e., if $\mathcal{C}_d=\mathcal{C}_{\rho}$.
\end{enumerate}
\end{definition}

\noindent \\ Clearly, segmental equivalence is stronger than convex equivalence. Moreover, the converse implication indeed fails, as can be seen from the following example:\\

\begin{example}\label{ex4.13}
Consider the connected graph $G$ represented by the diagram below:\\

\begin{figure}[H]
    \centering
    \includegraphics[width=6 cm]{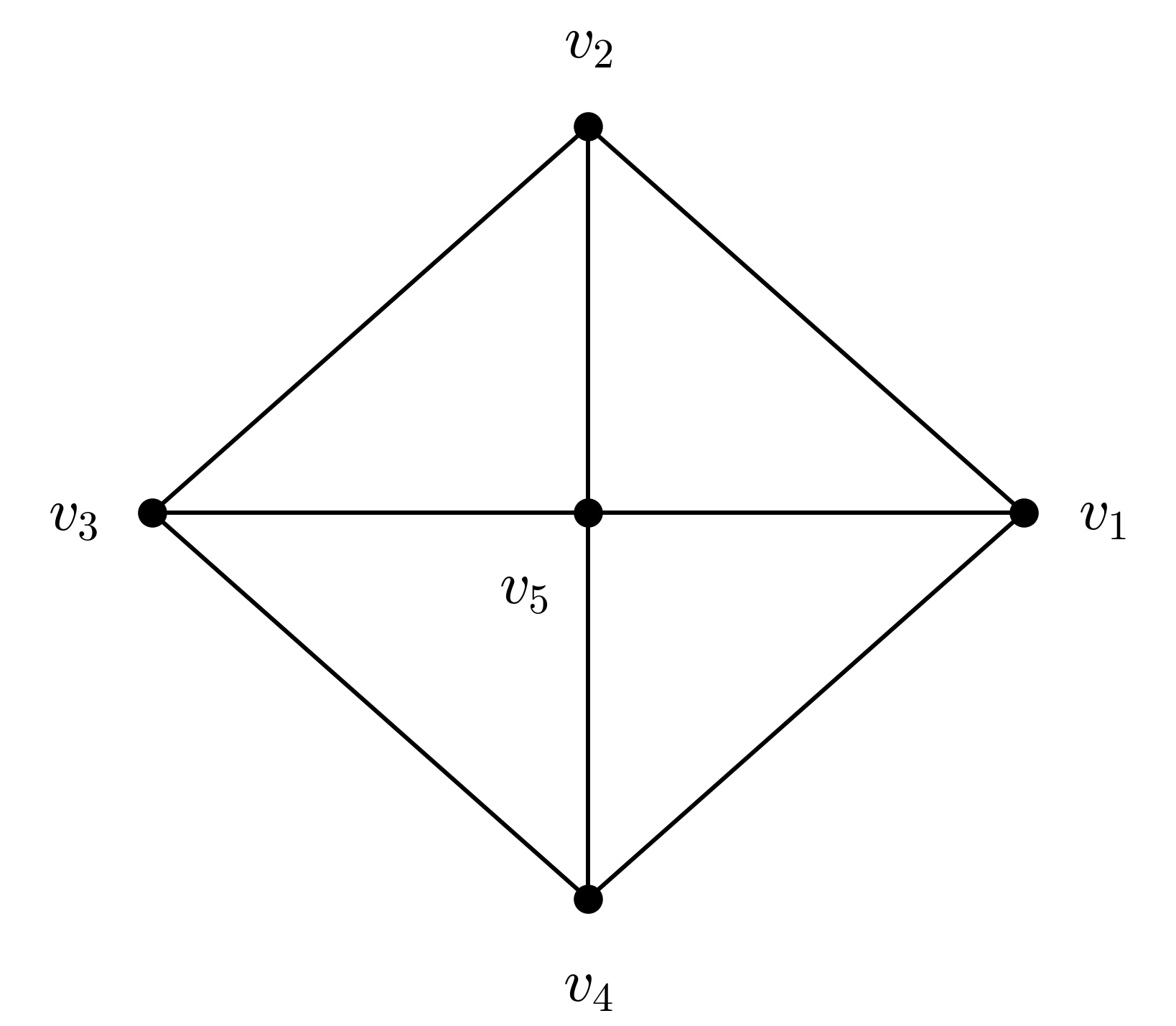}
   \caption{A connected graph on 5 vertices as a counterexample to segmen-\\
   tal equivalence.} \label{fig9}
    \end{figure}

\noindent \\ Take $d$ to be the usual geodesic metric on $G$. Furthermore, introduce weights on $G$ so that all edges have weight $2$ except the edge $v_2v_5$, which is to have weight $3$, and take $\rho$ to be the weighted geodesic metric on $G$ induced by this weighting. We claim that $d$ and $\rho$ are convexly equivalent. Indeed, one easily checks that under both metrics, the only non-trivial convex subsets of $G$ are the subsets corresponding to vertices, edge pairs, and the "central triangles" $\left\{ v_1, v_2, v_5\right\}$, $\left\{ v_2, v_3, v_5\right\}$, $\left\{ v_3, v_4, v_5\right\}$, and $\left\{ v_1, v_4, v_5\right\}$. However, $d$ and $\rho$ are not segmentally equivalent. Indeed, while $[v_2,v_4]_d=G$, we have that $[v_2,v_4]_{\rho}=G-\left\{ v_5 \right\}$.
\end{example}

\noindent \\ One easily observes that a pair of metrics $d, \rho$ on $X$ is convexly equivalent if and only if $\textup{conv}_{d}(\left\{ x,y\right\})=\textup{conv}_{\rho}(\left\{ x,y\right\})$ holds for all $x,y \in X$. Furthermore, when $d$ and $\rho$ induce convex metric segments, we have $\textup{conv}_{d}(\left\{ x,y\right\})=[x,y]_d$ and the analogous relation for $\rho$. Hence, by our previous remark, under such an assumption, the convex equivalence of $d$ and $\rho$ does imply the segmental equivalence of the metrics. By contrast, the metric $\rho$ on the graph in Example~\ref{ex4.13} does not induce convex metric segments (indeed, $[v_2,v_4]_{\rho}$ is non-convex), hence the counterexample. Additionally, we observe that convex equivalence is disjoint from topological equivalence. For instance, the $\ell^1$ and $\ell^2$ norms on $\mathbb{R}^2$ are topologically equivalent, but not convexly equivalent. Conversely, for the Euclidean metric $d$ on $\mathbb{R}$, the metric $\sqrt{d}$ is convexly equivalent to the discrete metric on $\mathbb{R}$, but not topologically equivalent.\\

\noindent Segmental and convex equivalence seem to be quite weak as notions of equivalence of metrics. For instance, as a consequence of Proposition~\ref{prop1.1}, one easily concludes that any pair of strictly convex norms induces segmentally equivalent metrics on the normed space. This makes it especially difficult to explicitly classify all metrics on a given set up to segmental or convex equivalence. However, we were successful in classifying the specialized case of metrics which are segmentally equivalent to the discrete metric. Namely, let us introduce the following definition:\\

\begin{definition}
A metric $d$ on a set $X$ is said to be \emph{segmentally discrete} if $d$ is segmentally equivalent to the discrete metric on $X$ i.e., if $[x,y]=\left\{ x,y\right\}$ holds for all $x,y \in X$.
\end{definition}

\noindent \\ One eventually observes that a large class of segmentally discrete metrics is given by $(\alpha-)$ \emph{snowflakes} i.e., metrics $d$ for which there exists a scalar $\alpha  \in (0,1)$ so that $d^{1\slash \alpha}$ is also a metric. Indeed, this follows as an immediate consequence of the  strict concavity of the map $t \mapsto t^{\alpha}$, as observed by Papadopoulos (see \cite{papadopoulos2005metric}). While not all segmentally discrete metrics are snowflakes, we were able to prove the following "local snowflaking" characterization of segmentally discrete metrics:\\

\begin{theorem}\label{thm4.15}
A metric $d$ on a set $X$ is segmentally discrete if and only if the restriction $d \mid_{F^2}$ is a snowflake for every non-empty finite subset $F \subseteq X$.
\end{theorem}

\vspace{0.3 cm}

\begin{proof}
Assume first that $d$ is segmentally discrete and let $F\subseteq X$ be finite and non-empty. To prove that $d^{1\slash \alpha}\mid_{F^2}$ is a metric for some $\alpha \in (0,1)$, it suffices to establish the triangle inequality for $d^{ 1\slash \alpha}\mid_{F^2}$. Thus, without loss of generality, we can assume that $\left | F\right |>2$. Furthermore, let $x,y,z \in F$ be distinct points. The segmental discreteness of $d$ gives $d(x,z)+d(z,y)>d(x,y)$. Additionally, by continuity, there exists $\varepsilon>0$ so that $d(x,z)^{1\slash \alpha}+d(z,y)^{1\slash \alpha}>d(x,y)^{1\slash \alpha}$ holds for all $1\slash(1+\varepsilon) < \alpha<1$. As $F$ is finite, it follows that there exists a global $\delta \in (0,1)$ so that $d^{1\slash \alpha}\mid_{F^2}$ satisfies the triangle inequality for all $\alpha \in (\delta, 1)$. Finally, a choice of $\alpha \in (\delta,1)$ ensures that  $d^{1\slash \alpha}\mid_{F^2}$ is a metric.\\

Conversely, assume that $d \mid_{F^2}$ is a snowflake for every non-empty finite subset $F \subseteq X$. Furthermore, let $x,y \in X$ and for the sake of contradiction, assume there exists $z \in [x,y]_{d}$ so that $z \notin \left\{ x,y\right\}$. Then, consider $F=\left\{ x,y,z\right\}$. By assumption, there exists $\alpha \in (0,1)$ so that $d^{ 1\slash \alpha}\mid_{F^2}$ is also a metric. Hence, the triangle inequality forces $d(x,z)^{1 \slash \alpha}+d(z,y)^{1\slash \alpha} \geq d(x,y)^{1 \slash \alpha}$. Now, the strict concavity of the map $t \mapsto t^{\alpha}$ implies

\begin{align*}
d(x,z)+d(z,y)>\left [ d(x,z)^{1 \slash \alpha}+d(z,y)^{1\slash \alpha} \right ]^{\alpha} \geq d(x,y),
\end{align*}

\noindent \\ which contradicts our assumption that $z \in [x,y]_{d}$. Hence, $[x,y]=\left\{ x,y\right\}$, and $d$ is segmentally discrete. This finishes the proof.
\end{proof}

\vspace{0.3 cm}

\noindent Additionally, we were able to make progress in the problem of classification of metrics up to segmental equivalence by developing the following method of generating segmentally equivalent metrics:\\

\begin{proposition}\label{prop4.16}
Let $d_1,...,d_n$ be pseudometrics on a set $X$ and let $\Phi, \Psi$ be segmentally equivalent norms whose restrictions to $\mathbb{R}_{\geq 0}^n$ are strictly monotone under the coordinatewise order. If the maps $d_{\Phi}, d_{\Psi}$ defined by the rules

\begin{align*}
d_{\Phi}: (x,y) &\mapsto \Phi(d_1(x,y),...,d_n(x,y))\\
d_{\Psi}: (x,y) &\mapsto \Psi(d_1(x,y),...,d_n(x,y))
\end{align*}

\noindent \\ are positive-definite, then $d_{\Phi}$ and $d_{\Psi}$ are segmentally equivalent metrics.
\end{proposition}

\vspace{0.3 cm}

\begin{proof}
Note first that $d_\Phi, d_\Psi$ are trivially symmetric and that the triangle inequality is forced by the monotonicity of the norms. As $d_\Phi, d_\Psi$ were assumed to additionally be positive-definite, we have that $d_{\Phi},d_{\Psi}$ are valid metrics on $X$. Furthermore, fix $x,y,z \in X$ and assume $z \in [x,y]_{d_{\Phi}}$. Under the labeling

\begin{align*}
\mathbf{p}&=(d_1(x,z),...,d_n(x,z))\\
\mathbf{q}&=(d_1(z,y),...,d_n(z,y))\\
\mathbf{r}&=(d_1(x,y),...,d_n(x,y)),\\
\end{align*}

\noindent our assumption is equivalent to the statement $\Phi(\mathbf{p})+\Phi(\mathbf{q})=\Phi(\mathbf{r})$. Applying the triangle inequality to each of the pseudometrics $d_i$, we have $\mathbf{r} \leq \mathbf{p}+\mathbf{q}$ coordinatewise. Now, for the sake of contradiction, assume $\mathbf{r} \neq  \mathbf{p}+\mathbf{q}$. Then, the strict monotonicity of $\Phi$ implies $\Phi(\mathbf{r})<\Phi(\mathbf{p}+\mathbf{q})$. Furthermore, the subadditivity of $\Phi$ gives $\Phi(\mathbf{p}+\mathbf{q})\leq \Phi(\mathbf{p})+\Phi(\mathbf{q}) $. Thus, $\Phi(\mathbf{r})<\Phi(\mathbf{p})+\Phi(\mathbf{q})$, contradicting our initial assumption. Hence, $\mathbf{p}+\mathbf{q}= \mathbf{r}$. It follows that $\Phi(\mathbf{p})+\Phi(\mathbf{q})=\Phi(\mathbf{p}+\mathbf{q})$, which is equivalent to the statement $\mathbf{p} \in  [\mathbf{0}, \mathbf{p} +\mathbf{q} ]_{\Phi}$. However, as $\Phi, \Psi$ were assumed to be segmentally equivalent norms, we also have that $\mathbf{p} \in  [\mathbf{0}, \mathbf{p} +\mathbf{q} ]_{\Psi}$, which is equivalent to the statement $\Psi(\mathbf{p})+\Psi(\mathbf{q})=\Psi(\mathbf{p}+\mathbf{q})$. Using the relation $\mathbf{p}+\mathbf{q}= \mathbf{r}$, we get that $\Psi(\mathbf{p})+\Psi(\mathbf{q})=\Psi(\mathbf{r})$, so $z \in [x,y]_{d_{\Psi}}$, and $[x,y]_{d_\Phi} \subseteq [x,y]_{d_\Psi}$. The reverse inclusion $[x,y]_{d_\Psi} \subseteq [x,y]_{d_\Phi}$ follows by symmetry. Thus, $d_{\Phi}$ and $d_{\Psi}$ are segmentally equivalent, as claimed.
\end{proof}

\noindent \\ Proposition~\ref{prop4.16} immediately implies the following useful corollaries:\\

\begin{corollary}\label{cor4.17}
Let $d_1,...,d_n$ be metrics on a set $X$ and let $w_1,...,w_n, w_1',...,w_n'>0$. Then, the metrics $d=\sum_{i=1}^{n}w_id_i$ and $d'=\sum_{i=1}^{n}w'_id_i$ are segmentally equivalent.
\end{corollary}

\vspace{0.3 cm}

\begin{proof}
Let $\Phi,\Psi: \mathbb{R}^n \rightarrow \mathbb{R}$ be maps defined by the rules

\begin{align*}
& \Phi: (p_1,...,p_n) \mapsto \left\| (w_1p_1,...,w_np_n)\right\|_1\\
&\Psi: (p_1,...,p_n) \mapsto \left\| (w'_1p_1,...,w'_np_n)\right\|_1,
\end{align*}

\noindent \\ where $\left\| \cdot \right\|_1$ is the $\ell^1$ norm on $\mathbb{R}^n$. Clearly, $\Phi,\Psi$ are norms on $\mathbb{R}^n$, and their restrictions to $\mathbb{R}_{\geq0}^n$ are strictly monotone.
Furthermore, we claim that $\Phi,\Psi$ are segmentally equivalent. By translation invariance of the norms, it sufficies to establish that $[\mathbf{0}, \mathbf{p}]_{\Phi}=[\mathbf{0}, \mathbf{p}]_{\Psi}$ for all $\mathbf{p}=(p_1,...,p_n) \in \mathbb{R}^n$. Thus, let $\mathbf{q}=(q_1,...,q_n) \in [\mathbf{0}, \mathbf{p}]_{\Phi}$. By definition, we have that $\Phi(\mathbf{q})+\Phi(\mathbf{p}-\mathbf{q})=\Phi(\mathbf{p})$, which is equivalent to the statement that

$$\sum_{i=1}^{n}w_i (\left | p_i-q_i\right |+\left | q_i\right |-\left | p_i\right |)=0.$$

\noindent \\ However, as each of the summands is non-negative, it follows that $\left | p_i-q_i\right |+\left | q_i\right |-\left | p_i\right |=0$ for all $1 \leq i \leq n$. Hence, we also have that

$$\sum_{i=1}^{n}w'_i (\left | p_i-q_i\right |+\left | q_i\right |-\left | p_i\right |)=0,$$

\noindent \\ which is equivalent to the statement $\Psi(\mathbf{q})+\Psi(\mathbf{p}-\mathbf{q})=\Psi(\mathbf{p})$. Thus, $\mathbf{q}\in [\mathbf{0}, \mathbf{p}]_{\Psi}$, and $[\mathbf{0}, \mathbf{p}]_{\Phi} \subseteq [\mathbf{0}, \mathbf{p}]_{\Psi}$. The reverse inclusion $[\mathbf{0},\mathbf{p}]_{\Psi} \subseteq [\mathbf{0}, \mathbf{p}]_{\Phi}$ follows by interchanging $\Phi, \Psi$. By our previous argument, we conclude that $\Phi$ and $\Psi$ are segmentally equivalent. Finally, our claim follows from Proposition~\ref{prop4.16} and the fact that $d=d_{\Phi}$ and $d'=d_{\Psi}$. 
\end{proof}

\vspace{0.3 cm}

\begin{corollary}\label{cor4.18}
Let $d_1,...,d_n$ be pseudometrics on a set $X$ and let $\Phi,\Psi$ be strictly convex norms whose restrictions to $\mathbb{R}_{\geq 0}^n$ are strictly monotone under the coordinatewise order. If the maps $d_{\Phi}, d_{\Psi}$ as defined in Proposition~\ref{prop4.16} are positive-definite, then $d_{\Phi}$ and $d_{\Psi}$ are segmentally equivalent metrics.
\end{corollary}

\vspace{0.3 cm}

\begin{proof}
By Proposition~\ref{prop1.1}, we easily conclude that $\Phi$ and $\Psi$ are segmentally equivalent. Our claim then follows from Proposition~\ref{prop4.16}.
\end{proof}

\noindent \\ For instance, an application of Corollary~\ref{cor4.18} gives that the $\ell^p$ sums

\begin{align*}
d&=\left ( \sum_{i=1}^{n}d_i^{p_1} \right )^{1\slash p_1}\\
d'&=\left ( \sum_{i=1}^{n}d_i^{p_2} \right )^{1\slash p_2}
\end{align*}

\vspace{0.3 cm}

\noindent of metrics $d_1,...,d_n$ are segmentally equivalent for all (finite) $p_1,p_2>1$ as a consequence of the strict convexity of the $\ell^{p}$ norm on $\mathbb{R}^n$ in the case that $p>1$. A similar and more useful consequence of Corollary~\ref{cor4.18} is the fact that the $\ell^p$ products

\begin{align*}
d_{p_1}: (x_1, ..., x_n), (y_1, ..., y_n)& \mapsto \left ( \sum_{i=1}^{n}d_i(x_i,y_i)^{p_1} \right )^{1\slash p_1}\\
d_{p_2}: (x_1, ..., x_n), (y_1, ..., y_n)& \mapsto \left ( \sum_{i=1}^{n}d_i(x_i,y_i)^{p_2} \right )^{1\slash p_2}
\end{align*}

\vspace{0.3 cm}

\noindent of metrics $d_1,...,d_n$ defined on sets $X_1,...,X_n$, respectively, are segmentally equivalent on $X=\prod_{i=1}^{n}X_i$ for all $p_1,p_2>1$. Consequently, $\mathcal{C}_{d_{p_1}}=\mathcal{C}_{d_{p_2}}$. Here, we will use $\mathcal{C}^{*}$ to  represent the convexity induced by any $\ell^p$ product of the metrics $d_1,...,d_n$ (for $p>1$).\\

\noindent On another note, the category $\textbf{Conv}$ is known to be topological over $\textbf{Set}$ (see \cite[Theorem 3.3]{wang2019categorical}) and, hence, bicomplete. Thus, additionally, we can consider the categorical product $\mathcal{C}=\prod_{i=1}^n \mathcal{C}_{d_i}$ of the individual convexities, which is generated by sets of the form $\prod_{i=1}^{n}C_i, \ C_i \in \mathcal{C}_{d_i}$ (see \cite[\S1]{van1993theory})\footnote{Here, by generating, we mean iterating arbitrary intersections and directed unions.}. Using the strict monotonicity of the $\ell^2$ norm on $\mathbb{R}^n$, one can easily observe that each generating set $\prod_{i=1}^{n}C_i$ is contained in $\mathcal{C}^*{}$. Consequently, $\mathcal{C}^{*}$ is finer than the product convexity $\mathcal{C}$. However, unlike the product convexity $\mathcal{C}$, the convexity $\mathcal{C}^{*}$ does not seem to be given by a universal property but, rather, appears to be functorial in nature i.e., dependent on the product in the category $\textbf{Met}$.

\vspace{0.3 cm}

\appendix
\section{Combinatorial sequences arising from $d$-convexity}

Given a ternary relation $B$ on a set $X$, it is natural to inquire whether $B$ can be realized as the betweenness relation induced by some metric on $X$. Such ternary relations are referred to as \emph{metrizable betweenness relations}. In a more general setting, one defines a \emph{betweenness relation} on $X$ to be a ternary relation $B$ satisfying expected geometric axioms such as\\

\begin{enumerate}
\itemsep0.3em
\item[(B0)] $B(x,y,x) \implies x=y$. \label{b0}
\item[(B1)] $B(x,x,y)$ \label{b1}
\item[(B2)] $B(x,y,z)\implies B(z,y,x)$ \label{b2}
\item[(B3)] $B(x,y,z) \wedge B(x,z,u) \implies B(x,y,u)$ \label{b3}
\item[(B4)] $B(x,y,z) \wedge B(x,z,u) \implies B(y,z,u)$ \label{b4}
\end{enumerate} 

\noindent \\ for all $x,y,z,u 
\in X$ (see \cite{mendris1995axiomatization}). Interpreting the problem as a linear feasibility problem and employing linear programming, we enumerated metrizable betweenness relations on finite sets. Let $b_n$ denote the number of (labeled) metrizable betweenness relations on an $n$-set and let $\overline{b}_n$ denote the number of such relations up to isomorphism. The first values are given in the table below:\\

\renewcommand{\arraystretch}{1.15}

\begin{table}[H]
\centering
\begin{tabular}{c|cccccc}
$n$ & 1 & 2 & 3 & 4 & 5 & 6\\
\hline
$b_n$ & 1 & 1 & 4 & 74 & 8628 & 7238428\\
$\overline{b}_n$ & 1 & 1 & 2 & 9 & 122 & $\geq 6778$
\end{tabular}
\caption{Number of (nonisomorphic) metrizable betweenness relations on $n$-sets.}
\label{table1}
\end{table}

\noindent \\ We have obtained the estimate $\overline{b}_6 \geq 6778$ using Monte Carlo sampling, but the exact value of $\overline{b}_6$ is currently unknown to us. The sequences $b_n$ and $\overline{b}_n$ have been entered into the OEIS as A395237 (\cite{A395237}) and A397184
(\cite{A397184}), respectively. Additionally, observe that $b_n$ also counts the number of segmental equivalence classes of metrics on an $n$-set. Indeed, metric segments in any metric space $(X,d)$ are uniquely determined by the induced betweenness relation $B$ via the identity $[x,y]=\left\{ z \in X: B(x,z,y)\right\}$ for all $x,y \in X$.\\

\noindent Our observation also allowed us to count metrizable convexities on finite sets. Formally, we say that a convexity $\mathcal{C}$ on a set $X$ is \emph{metrizable} when a metric $d$ can be defined on $X$ so that $\mathcal{C}=\mathcal{C}_d$. Specifically, we started from the counts of metrizable betweenness relations given in Table~\ref{table1}, evaluated the associated metric segments, and computed the induced convexities. Let $c_n$ denote the number of (labeled) metrizable convexities on an $n$-set and let $\overline{c}_n$ denote the number of such convexities up to isomorphism. The first values are given in the table below:\\

\renewcommand{\arraystretch}{1.15}

\begin{table}[H]
\centering
\begin{tabular}{c|cccccc}
$n$ & 1 & 2 & 3 & 4 & 5 & 6\\
\hline
$c_n$ & 1 & 1 & 4 & 74 & 8508 & 6460153\\
$\overline{c}_n$ & 1 & 1 & 2 & 9 & 119 & $\geq 6778$
\end{tabular}
\caption{Number of (nonisomorphic) metrizable convexities on $n$-sets.}
\label{table2}
\end{table}

\noindent \\ Analogously to $\overline{b}_6$, we have obtained the estimate $\overline{c}_6 \geq 6778$ using Monte Carlo sampling, but the exact value of $\overline{c}_6$ is currently unknown to us. The sequences $c_n$ and $\overline{c}_n$ have been entered into the OEIS as A395250 (\cite{A395250}) and A395485
(\cite{A395485}), respectively. Trivially, $c_n$ also counts the number of convex equivalence classes of metrics on an $n$-set.\\

\noindent We refer the interested reader to the works of Mendris \& Zlatoš (see \cite{mendris1995axiomatization}) and Szabó (see \cite{szabo2022betweenness}), who study metrizability of betweenness relations from the standpoints of logic and graph theory.

\vspace{0.3 cm}

\section*{Acknowledgments}
The author is indebted to Professor Deping Ye and Mario Stipcic for giving him an opportunity to hold talks on this work. The author is also thankful to Professor Rolf Schneider and Eike Hertel for encouraging the publication of this paper.

\vspace{0.3 cm}

\bibliographystyle{amsplain}
\bibliography{references}
\end{document}